\documentclass[a4paper]{article}
\usepackage{amssymb,amsmath}
\usepackage[latin2]{inputenc}

\usepackage{amssymb}
\usepackage{epstopdf}

\def\cocoa{{\hbox{\rm C\kern-.13em o\kern-.07em C\kern-.13em o\kern-.15em A}}}

\newtheorem{theorem}{Theorem}[section]

\newtheorem{lemma}[theorem]{Lemma}
\newtheorem{proposition}[theorem]{Proposition}
\newtheorem{remark}[theorem]{Remark}

\newtheorem{question}[theorem]{Question}

\newenvironment{proof}{\begin{trivlist}\item[]{\it
Proof.}}{\hfill$\square$\end{trivlist}}

\begin{document}
\title{Discriminant of symmetric matrices as a sum of squares and the orthogonal group} 
\author{M\'aty\'as Domokos\thanks{Partially supported by 
OTKA NK72523, NK81203.} 
\\
{\small R\'enyi Institute of Mathematics, Hungarian Academy of 
Sciences,} 
\\ {\small P.O. Box 127, 1364 Budapest, Hungary,} 
{\small E-mail: domokos.matyas@renyi.mta.hu } 
}
\date{}
\maketitle 

\begin{abstract}
It is proved that the discriminant of $n\times n$ real symmetric matrices can be written as a sum of squares, where the number of summands equals the dimension of the space of $n$-variable spherical harmonics of degree $n$. 
The representation theory of the orthogonal group is applied to express the discriminant of three by three real symmetric matrices as a sum of five squares,  and to show that it can not be written as  the sum of less than five squares. It is proved that the discriminant of four by four real symmetric matrices can be written as a sum of seven squares. 
These improve results of Kummer from 1843 and Borchardt from 1846. 
\end{abstract}

\section{Introduction}\label{sec:intro}

The discriminant of a degree $n$ monic polynomial $p$ with roots $\lambda_1,\ldots,\lambda_n$ equals 
\[\prod_{1\leq i<j\leq n}(\lambda_i-\lambda_j)^2.\] 
Recall that it can be written as a polynomial function of  the coefficients of $p$.  
By the {\it discriminant $\delta(A)$ of an $n\times n$ real symmetric matrix} $A$ we mean the discriminant of the characteristic polynomial  of $A$. 
Recall that $\delta(A)$ is a homogeneous polynomial function in the entries of $A$, of degree 
$n(n-1)$. Moreover, $\delta(A)=0$ if and only if the matrix $A$ is degenerate (i.e. has multiple eigenvalues). 

Denote by ${\mathcal{M}}$ the space of $n\times n$ real symmetric matrices ($n\geq 2$). It is a vector space of dimension 
$n(n+1)/2$ over ${\mathbb{R}}$. 
It contains the subset ${\mathcal{E}}$ of degenerate real symmetric matrices, a real algebraic subvariety. Although ${\mathcal{E}}$ is the zero locus of the single polynomial $\delta\in{\mathbb{R}}[{\mathcal{M}}]$, it has codimension two in 
${\mathcal{M}}$ by a result of Neumann and Wigner (cf. \cite{lax:1997}). An algebraic explanation of the fact that 
the codimension is greater than $1$ is that $\delta$ can be written as a sum of squares in ${\mathbb{R}}[{\mathcal{M}}]$. 
An explicit presentation of $\delta$ as a sum of seven squares was given in the nineteenth century by Kummer   \cite{kummer} for $n=3$ (see Remark~\ref{remark:kummer} for details).  
Borchardt \cite{borchardt} showed that $\delta$ is a sum of squares for arbitrary $n$. More recent approaches to this result are due to Newell \cite{newell}, Ilyushechkin \cite{ilyushechkin}, 
Lax \cite{lax:1998}, Parlett \cite{parlett}. 
Denote by $\mu(n)$ the minimal number of summands in a representation of $\delta$ as a sum of squares. The exact value of $\mu(n)$  is known only for $n\leq 3$: a straightforward calculation yields $\mu(2)=2$,  and the equality $\mu(3)=5$ will be proved here. 

The approach of Lax \cite{lax:1998} to this problem is substantially different from the other works mentioned above, as it makes a crucial use of the conjugation action of the orthogonal group on the space of symmetric matrices. The present paper develops further the key ideas from 
\cite{lax:1998}, by exploiting deeper the representation theory of the orthogonal group.  
Our main results are the following. Theorem~\ref{thm:spherical}  asserts that the degree $n(n-1)/2$ homogeneous component of the vanishing ideal of the variety of degenerate real symmetric $n\times n$ matrices contains an $ SO_n$-submodule isomorphic to the space 
${\mathcal{H}}^n({\mathbb{R}}^n)$ of $n$-variable spherical harmonics of degree $n$. 
Consequently, the discriminant can be written as the sum of 
$\dim({\mathcal{H}}^n({\mathbb{R}}^n))=\binom{2n-1}{n-1}-\binom{2n-3}{n-1}$ squares. 
Note that the number $\dim({\mathcal{H}}^3({\mathbb{R}}^3))=7$ agrees with the bound for $\mu(3)$ obtained by Kummer (or recently by Parrilo \cite{parrilo} using an algorithm based on semidefinite programming). However, aided by representation theory we derive an explicit presentation of 
the discriminant of $3\times 3$ symmetric matrices  as the sum of five squares 
(we learnt after submission that essentially the same identity was found earlier by 
Watson \cite{watson} in a different way),  and prove that $\mu(3)=5$, see 
Theorem~\ref{thm:fivesquares}; this 
shows that $\mu(n)$ may be strictly smaller than the minimal dimension of an $SO_n$-invariant subspace in the degree $n(n-1)/2$ homogeneous component of the vanishing ideal of degenerate symmetric matrices. 
In Theorem~\ref{thm:n=4} we locate some further irreducible $ SO_4$-module summands in the degree six homogeneous component of the vanishing ideal of degenerate symmetric $4\times 4$ matrices, and conclude that the discriminant in this case can be written as the sum of seven squares. 

The paper is organized as follows. First in Section~\ref{sec:general} we place the problem into a more general context, relating degeneracy loci of linear actions of compact Lie groups, and point out the existence of some generalized "discriminants" that are sums of squares by basic principles of representation theory; this is done mainly for sake of completeness of the picture, only 
Lemma~\ref{lemma:onb} is logically necessary for the rest of the paper. Using an observation of Lax \cite{lax:1998} (see Lemma~\ref{lemma:lax}), we conclude in Section~\ref{sec:degsymmat} that $\mu(n)$ is bounded by the minimal dimension of an $SO_n$-invariant subspace in the degree $n(n-1)/2$ homogeneous component of the vanishing ideal of degenerate symmetric $n\times n$ matrices, see Lemma~\ref{lemma:irred}. 
Section~\ref{sec:cov} contains a crucial new ingredient in our work: an explicit construction of an 
$O_n$-module homomorphism ${\mathcal{T}}^\star$ from the $(n-1)$th exterior power of the space of trace zero symmetric $n\times n$ matrices into the degree $n(n-1)/2$ homogeneous component of the vanishing ideal of degenerate matrices.  In Section~\ref{sec:repsofonr} we recall the facts from the representation theory of the orthogonal group that we shall use. 
Our best general upper bound for $\mu(n)$ (cf. Theorem~\ref{thm:spherical}) is discussed in Section~\ref{sec:spherical}. 
The results on the  $3\times 3$ and $4\times 4$ cases are contained in 
Sections~\ref{sec:n=3} and \ref{sec:n=4}. In Section~\ref{sec:kercovstar} for general $n$ we locate some irreducible $SO_n$-module summands in the kernel of ${\mathcal{T}}^\star$; in the special cases $n=3,4$ these results are used in the previous two sections. 

We finish the Introduction by mentioning two related subsequent developments, partially inspired by the present work: M. Ra{\"\i}s \cite{rais} informed us about a generalization of the discriminant (and its representation as a sum of squares) in the context of reductive Lie algebras and their Cartan decomposition. C. Gorodski \cite{gorodski} generalizes our Theorem~\ref{thm:spherical} in the context of symmetric spaces.  (We learnt  also from \cite{gorodski} the references \cite{watson} and \cite{newell}, 
extending our account of earlier works on representing the discriminant of symmetric matrices as a sum of squares.) 


\section{General discriminants}\label{sec:general} 

Let $G$ be a compact Lie group (over ${\mathbb{R}}$) with a (smooth) representation on the finite dimensional real vector space $V$. Denote by ${\mathbb{R}}[V]$ the algebra of real valued polynomial functions on $V$ (the {\it coordinate ring of $V$} in the terminology of algebraic geometry). There is an induced representation of $G$ on ${\mathbb{R}}[V]$ 
given by the formula $(g\cdot f)(v):=f(g^{-1}v)$ ($g\in G$, $v\in V$, $f\in {\mathbb{R}}[V]$).  
Note that the homogeneous components of ${\mathbb{R}}[V]$ are $G$-stable. 
By a {\it $G$-invariant} we mean a polynomial $f\in{\mathbb{R}}[V]$ with $g\cdot f=f$ for all $g\in G$. 
The following statement is just a reformulation of the well known fact  that since $G$ is compact, for any finite dimensional representation of $G$ there exists a $G$-invariant symmetric, positive definite bilinear form on the underlying vector space. 

\begin{lemma}\label{lemma:onb} 
Any finite dimensional $G$-stable subspace $W$ in ${\mathbb{R}}[V]$ has  
an ${\mathbb{R}}$-basis $h_1,\ldots,h_n$ such that $H:=\sum_{i=1}^nh_i^2$ is $G$-invariant. 
Moreover, the zero locus of the polynomial $H$ in $V$ coincides with the common zero locus of  the elements of $W$. 
\end{lemma} 

For a positive integer  $r$ set  
$V^{<r}:=\{v\in V\mid \dim(G\cdot v)<r\}$. 
Note that the dimension of any orbit $G\cdot v$ is less than or equal to $\dim(G)$. 

\begin{proposition}\label{prop:generaldiscriminant} Let $d$ denote the maximal dimension of an orbit in $V$.  For $r=1,\ldots,d$, there exists a non-zero $G$-invariant polynomial function $D_r$ on $V$, homogeneous of degree $2r$, such that $D_r(v)=0$ if and only if $v\in V^{<r}$. 
Moreover, $D_r$ is the sum of squares of homogeneous polynomials of degree $r$,  
and the number of summands is less than or equal to 
$\binom{\dim(G)}{r}\cdot\binom{\dim(V)}{r}$ 
(product of binomial coefficients). 
\end{proposition}

\begin{proof} Fix $r\in\{1,\ldots,d\}$. One can find polynomial equations on $V$ whose common zero locus is $V^{<r}$ as follows. The dimension of the orbit of $v\in V$ is the  difference of the dimension of $G$ and the dimension of the stabilizer subgroup $G_v$ of $v$ in $G$. One can pass to the tangent representation of the 
Lie algebra ${\mathrm{Lie}}(G)$ on $V$, this is a Lie algebra homomorphism 
$\tau:{\mathrm{Lie}}(G)\to {\mathrm{gl}}(V)$ (where ${\mathrm{gl}}(V)$ is the Lie algebra of all linear transformations of $V$). The dimension of $G_v$ is the same as the dimension of its Lie algebra. Now 
\[{\mathrm{Lie}}(G_v)=\{A\in {\mathrm{Lie}}(G)\mid \tau(A)v=0\}\] 
For each $v\in V$ denote by $L_v$ the ${\mathbb{R}}$-linear map ${\mathrm{Lie}}(G)\to V$ given by $L_v(A):=\tau(A)v$ for 
$A\in{\mathrm{Lie}}(G)$. 
Note that the map $V\to \hom_{{\mathbb{R}}}({\mathrm{Lie}}(G),V)$, $v\mapsto L_v$ is ${\mathbb{R}}$-linear. 
Moreover, $\ker(L_v)={\mathrm{Lie}}(G_v)$, hence the dimension of the orbit of $v$ equals the dimension of the image of $L_v$. 
Thus $v$ belongs to $V^{<r}$ if and only if the rank of $L_v$ is less than $r$. 
Fixing a basis in $V$ and in ${\mathrm{Lie}}(G)$, $L_v$ is identified with a matrix of size $\dim(V)\times \dim(G)$, 
whose $(i,j)$-entry equals $\xi_{ij}(v)$, where $\xi_{ij}$ are linear forms on $V$. 
Consequently, $V^{<r}$ is the common zero locus of the determinants of the $r\times r$-minors of 
the matrix $(\xi_{ij})$. The determinant of an $r\times r$ minor is a degree $r$ homogeneous  
element in ${\mathbb{R}}[V]$ (unless it is the zero polynomial). The assumption $r\leq d$ implies that not all of these determinants are identically zero. 

Next we show that these degree $r$ homogeneous polynomials span a $G$-stable subspace in ${\mathbb{R}}[V]$. Indeed, it is easy to see that  
\[L_{g^{-1}v}=T(g^{-1})\circ L_v \circ {\mathrm{ad}}(g)\]
where ${\mathrm{ad}}:G\to GL({\mathrm{Lie}}(G))$ denotes the adjoint representation of $G$ on its Lie algebra, and 
$T:G\to GL(V)$ is the given representation of $G$ on $V$. Consequently, the image under $g\in G$ of the set of determinants of the  $r\times r$ minors of $(\xi_{ij})_{i=1,\ldots,\dim(V)}^{j=1,\ldots,\dim(G)}$ is the set of determinants of the $r\times r$ minors of $P(\xi_{ij})Q$, where $P,Q$ are the matrices of 
$T(g^{-1})$, ${\mathrm{ad}}(g)$ with respect to the chosen bases of $V$, ${\mathrm{Lie}}(G)$. 
By the Binet-Cauchy formula we conclude that the determinants of the $r\times r$ minors of $(\xi_{ij})$ span a $G$-stable subspace in ${\mathbb{R}}[V]$. 

Now the Proposition follows from Lemma~\ref{lemma:onb}. 
\end{proof} 

\begin{remark} \label{remark:special} An example of the general setup discussed above is the case of the orthogonal group $G:= O_n$ acting on the space $V:={\mathcal{M}}$ of real symmetric $n\times n$ matrices by 
conjugation: for $g\in  O_n$ and $A\in{\mathcal{M}}$ we have $g\cdot A:=gAg^T$ (matrix multiplication). 
Then the stabilizer of a matrix with distinct eigenvalues is zero-dimensional, so the maximal dimension of an orbit is $\dim( O_n)=n(n-1)/2$, and 
${\mathcal{E}}:=V^{<n(n-1)/2}$ is the set of degenerate matrices. 
Moreover, the Lie algebra of $ O_n$ is the space of skew-symmetric matrices (with the commutator as the Lie bracket). 
So in this case we get back exactly the polynomials vanishing on ${\mathcal{E}}$ that are constructed in 
\cite{lax:1998}. 
By Proposition~\ref{prop:generaldiscriminant} we conclude the existence of a degree 
$n(n-1)$ homogeneous $ O_n$-invariant $D$, which is a sum of squares, 
and the zero locus of $D$ is 
${\mathcal{E}}$. By an observation of Lax (see Lemma~\ref{lemma:lax} below) $D$ must coincide with  a positive scalar multiple of $\delta$. 
\end{remark} 

The above considerations motivate the following general question:    

\begin{question} 
When do the polynomials constructed in the proof of Proposition~\ref{prop:generaldiscriminant} generate the vanishing ideal of  $V^{<r}$? 
 \end{question} 
 

\section{Bounding $\mu(n)$ with the dimension of some irreducible representation}\label{sec:degsymmat}

As a representation of $ O_n$, the space 
${\mathcal{M}}$ decomposes as 
\begin{equation}\label{eq:decomp}
{\mathcal{M}}={\mathbb{R}} I\oplus {\mathcal{N}}
\end{equation} 
where ${\mathcal{N}}$ stands for the codimension one subspace of trace zero matrices, and $I$ is the 
$n\times n$ identity matrix. The projection onto the second direct summand in (\ref{eq:decomp}) identifies ${\mathbb{R}}[{\mathcal{N}}]$ with an $ O_n$-stable subalgebra of ${\mathbb{R}}[{\mathcal{M}}]$.  
Moreover, ${\mathbb{R}}[{\mathcal{M}}]={\mathbb{R}}[{\mathcal{N}}][{\mathrm{Tr}}]$ is a polynomial ring over ${\mathbb{R}}[{\mathcal{N}}]$ generated by the trace function ${\mathrm{Tr}}:{\mathcal{M}}\to{\mathbb{R}}$ (which is $ O_n$-invariant).  
Write ${\mathcal{F}}:={\mathcal{E}}\cap{\mathcal{N}}$ for the set of degenerate trace zero symmetric matrices. 
Denote by ${\mathcal{I}}({\mathcal{E}})$ the ideal in ${\mathbb{R}}[{\mathcal{M}}]$ consisting of the polynomials that vanish on 
${\mathcal{E}}$. 
Similarly, ${\mathcal{I}}({\mathcal{F}})$ stands for the vanishing ideal of ${\mathcal{F}}$ in 
${\mathbb{R}}[{\mathcal{N}}]$.  Obviously we have ${\mathcal{E}}={\mathbb{R}} I\oplus {\mathcal{F}}$, implying 
\begin{equation}\label{eq:ideals}
{\mathcal{I}}({\mathcal{E}})={\mathcal{I}}({\mathcal{F}})[{\mathrm{Tr}}].
\end{equation} 
Therefore the study of ${\mathcal{I}}({\mathcal{E}})$ is essentially equivalent to the study of ${\mathcal{I}}({\mathcal{F}})$. 

The definition of the discriminant in terms of the eigenvalues implies that $\delta$ belongs to the subalgebra ${\mathbb{R}}[{\mathcal{N}}]$ of ${\mathbb{R}}[{\mathcal{M}}]$. From now on we shall focus on the algebra ${\mathbb{R}}[{\mathcal{N}}]$ and the ideal ${\mathcal{I}}({\mathcal{F}})$.

\begin{remark}\label{remark:tracezero} 
Replacing ${\mathcal{M}}$ by its subspace ${\mathcal{N}}$ in the argument in Remark~\ref{remark:special}  
we get the bound  $\mu(n)\leq \binom{n(n+1)/2-1}{n-1}$. 
This is already better than the bound 
$\binom{\binom{n+1}{2}}{n}-\binom{\binom{n}{2}}{n}-(n-1)\binom{\binom{n}{2}}{n-1}$ 
that can be infered from  Section 4.3 of \cite{parlett}. 
For $n=4$ we get  $\mu(4)\leq 84$, a result stated without proof by Borchardt \cite{borchardt}. For $n=3$ we get  $\mu(3)\leq 10$. 
However, the better bound $\mu(3)\leq 7$ was obtained in \cite{kummer}. 
This will be generalized for arbitrary $n$ in Section~\ref{sec:spherical}.   
\end{remark}

Restricting to diagonal matrices one easily shows that no polynomial of degree less than $n(n-1)/2$ vanishes on ${\mathcal{F}}$. 
By Proposition~\ref{prop:generaldiscriminant} (and the explanation afterwards), the degree $n(n-1)/2$ homogeneous component ${\mathcal{I}}({\mathcal{F}})_{n(n-1)/2}$
of the ideal ${\mathcal{I}}({\mathcal{F}})$ is non-zero; even more, the common zero locus of 
${\mathcal{I}}({\mathcal{F}})_{n(n-1)/2}$ is ${\mathcal{F}}$. 
Moreover, ${\mathcal{I}}({\mathcal{F}})_{n(n-1)/2}$
is an $ O_n$-submodule in ${\mathbb{R}}[{\mathcal{N}}]$. Indeed, 
the subset ${\mathcal{F}}$ of ${\mathcal{N}}$ is $ O_n$-stable, hence the ideal ${\mathcal{I}}({\mathcal{F}})$ is an $ O_n$-submodule of ${\mathbb{R}}[{\mathcal{N}}]$. Since the action of $ O_n$ preserves the grading on ${\mathbb{R}}[{\mathcal{N}}]$, the homogeneous component  ${\mathcal{I}}({\mathcal{F}})_{n(n-1)/2}$ is an $ O_n$-submodule. 
To get the best bounds on $\mu(n)$, we shall switch from $ O_n$ to its subgroup $ SO_n$ consisting of the orthogonal matrices with determinant one ($ SO_n$ is callled the {\it special orthogonal 
group}).

\begin{lemma}\label{lemma:irred} 
Any non-zero $ SO_n$-submodule of ${\mathcal{I}}({\mathcal{F}})_{n(n-1)/2}$ has a basis $\{f_i\}$ such that $\delta=\sum f_i^2$. 
Consequently, $\mu(n)$ is less than or equal to the minimal dimension of an irreducible $ SO_n$-submodule contained in 
${\mathcal{I}}({\mathcal{F}})_{n(n-1)/2}$. 
\end{lemma} 

\begin{proof} Let $W$ be a non-zero $ SO_n$-invariant subspace of the degree $n(n-1)/2$ homogeneous component of ${\mathcal{I}}({\mathcal{F}})$. By Lemma~\ref{lemma:onb}, $W$ has a basis 
$\{f_i\}$ such that $D:=\sum_if_i^2$ is $ SO_n$-invariant (non-zero by positivity of the summands). 
Moreover, $\deg(D)=n(n-1)$ and $D$ vanishes on ${\mathcal{F}}$. 
By Lemma~\ref{lemma:lax} below, $D=c\delta$ for some positive scalar $c\in{\mathbb{R}}$, so $\delta=\sum_i(c^{-1/2}f_i)^2$. 
\end{proof}

The following lemma is due to Lax; it is stated in \cite{lax:1998} for ${\mathcal{M}}$ and ${\mathcal{E}}$, but obviously holds in the form below by (\ref{eq:ideals}):

\begin{lemma}\label{lemma:lax} 
Up to scalar multiples, $\delta$ is the only degree $n(n-1)$ homogeneous $ SO_n$-invariant polynomial function on ${\mathcal{N}}$ that vanishes on ${\mathcal{F}}$. 
\end{lemma} 


\section{An $ O_n$-module homomorphism  into ${\mathcal{I}}({\mathcal{F}})$}\label{sec:cov}   

Next we turn to a crucial step in the present paper, and provide a simple construction of a non-zero  $ O_n$-module homomorphism into ${\mathcal{I}}({\mathcal{F}})_{n(n-1)/2}$.  
One has the $ O_n$-equivariant polynomial maps 
\begin{equation}\label{eq:H_i}H_i:{\mathcal{N}}\to {\mathcal{N}},\quad  
A\mapsto A^i-\frac{1}{n}{\mathrm{Tr}}(A^i)I\end{equation}  
for $i=1,2,\ldots$. 
Using them one defines a map from ${\mathcal{N}}$ to the degree $n-1$ exterior power of ${\mathcal{N}}$: 
\begin{equation}\label{eq:T}
{\mathcal{T}} :{\mathcal{N}}\to \bigwedge^{n-1}{\mathcal{N}}, \quad 
A\mapsto A\wedge H_2(A)\wedge\cdots\wedge H_{n-1}(A) \end{equation} 

\begin{proposition}\label{prop:cov} 
For $A\in{\mathcal{N}}$ we have ${\mathcal{T}}(A)=0$ if and only if $A$ belongs to ${\mathcal{F}}$. 
\end{proposition} 

\begin{proof} Denote by ${\mathcal{D}}$ the space of trace zero diagonal matrices, ${\mathcal{D}}_1$ the subspace of ${\mathcal{D}}$ consisting of the matrices whose first two diagonal entries coincide, and 
${\mathcal{D}}_0$ the subset of matrices with distinct diagonal entries. Clearly the $H_j$ map ${\mathcal{D}}_1$ into itself, so ${\mathcal{T}}({\mathcal{D}}_1)\subseteq\bigwedge^{n-1}{\mathcal{D}}_1=0$, since $\dim_{{\mathbb{R}}}({\mathcal{D}}_1)<n-1$. 
On the other hand, we claim that for $A\in{\mathcal{D}}_0$ the $H_j(A)$ $(j=1,\ldots,n-1)$ are linearly independent, and therefore span ${\mathcal{D}}$. Indeed, $A\in{\mathcal{D}}_0$ has distinct diagonal entries 
$a_1,\ldots,a_n$. Consider the Vandermonde matrix $V:=(a_i^{j-1})_{i,j=1}^n$, its columns are linearly independent. Denote by $V'$ the matrix obtained from $V$ by subtracting from the $j$th column of $V$ the first column of $V$ multiplied by 
$1/n\sum_{i=1}^na_i^{j-1}$, for $j=2,\ldots,n$. Clearly, the columns of $V'$ are linearly independent. Since $H_{j-1}(A)$ can be identified with the $j$th column of $V'$ for $j=2,\ldots,n$, our claim folows. Consequently, ${\mathcal{T}}(A)\in\bigwedge^{n-1}{\mathcal{D}}$ is non-zero. 

If $A\in{\mathcal{F}}$, then 
the $O_n$-orbit of $A$ intersects ${\mathcal{D}}_1$, therefore by $O_n$-equivariance of ${\mathcal{T}}$ we conclude that ${\mathcal{T}}(A)=0$. 
Similarly, if $A\in{\mathcal{N}}\setminus{\mathcal{F}}$, then the $O_n$-orbit of $A$ intersects ${\mathcal{D}}_0$, 
consequently ${\mathcal{T}}(A)\neq 0$. 
\end{proof}

Now ${\mathcal{T}}$ induces a non-zero $ O_n$-equivariant linear map ${\mathcal{T}}^\star$ from the dual 
space of $\bigwedge^{n-1}{\mathcal{N}}$ defined as follows: 
\begin{align*}
{\mathcal{T}}^\star:(\bigwedge^{n-1}{\mathcal{N}})^\star\to{\mathcal{I}}({\mathcal{F}})_{n(n-1)/2}\\
({\mathcal{T}}^\star(\xi))(A):=\xi({\mathcal{T}}(A))\mbox{ for }\xi\in(\bigwedge^{n-1}{\mathcal{N}})^\star,A\in{\mathcal{N}}   
\end{align*}
Indeed, ${\mathcal{T}}$ is a polynomial map, and in the terminology of algebraic geometry, ${\mathcal{T}}^\star$ is the restriction to $(\bigwedge^{n-1}{\mathcal{N}})^\star\subset{\mathbb{R}}[\bigwedge^{n-1}{\mathcal{N}}]$ of the comorphism of the morphism ${\mathcal{T}}$ of affine algebraic varieties. Since the polynomial map ${\mathcal{T}}$ is homogeneous of degree 
$1+2+\cdots+(n-1)=n(n-1)/2$, the image of ${\mathcal{T}}^\star$ is contained in the degree $n(n-1)/2$ 
homogeneous component of ${\mathbb{R}}[{\mathcal{N}}]$. Since ${\mathcal{T}}$ is $ O_n$-equivariant, the same holds for ${\mathcal{T}}^\star$. By Proposition~\ref{prop:cov}, the image of ${\mathcal{T}}^\star$ is a subspace of 
${\mathcal{I}}({\mathcal{F}})$, furthermore, the common zero locus in ${\mathcal{N}}$ of the polynomials from the image of ${\mathcal{T}}^\star$ is ${\mathcal{F}}$. 
In particular, ${\mathcal{T}}^\star$ is non-zero. 


\section{Representations of $ O_n$}\label{sec:repsofonr} 

A classical reference for the material in this section is \cite{weyl}; see also \cite{procesi}, 
\cite{goodman-wallach}, \cite{fulton-harris} for more modern treatments. 
By a {\it representation} of $ O_n$ (resp. $ SO_n$) we mean a Lie group homomorphism from 
$ O_n$ (or $ SO_n$) into the real Lie group of all linear transformations of a finite dimensional vector space over the field of real numbers. Since these groups are compact, all representations decompose as a sum of irreducibles, and all representations are self-dual.  
The irreducible representations of $ O_n$ and $ SO_n$ all appear as summands in the tensor powers of the defining representation of $ O_n$ on ${\mathbb{R}}^n$ (see \cite{weyl}), and  
the isomorphism classes of irreducible representations of $ O_n$ are  traditionallly labeled by partitions.  
By a partition $\lambda=(\lambda_1,\ldots,\lambda_n)$ we mean a decreasing sequence 
$\lambda_1\geq\cdots\geq\lambda_n\geq 0$ of non-negative integers. 
For $j=1,2,\ldots$, set $h_i(\lambda):=|\{j \mid \lambda_j\geq i\}|$ (the length of the $i$th column of the Young diagram of $\lambda$). The isomorphism classes of irreducible representations of $ O_n$ are in bijection with partitions $\lambda$  satisfying $h_1(\lambda)+h_2(\lambda)\leq n$ (see for example Section 6.5 in \cite{procesi}). Denote by $V_{\lambda}$ the irreducible $ O_n$-module corresponding to $\lambda$. 
For the partition $(d)=(d,0,\ldots,0)$ we have that 
$V_{(d)}\cong {\mathcal{H}}^d({\mathbb{R}}^n)$, the space of spherical harmonics of degree $d$ in $n$ variables. 
It can be constructed as follows: consider the natural representation of $ O_n$ on the coordinate ring ${\mathbb{R}}[x_1,\ldots,x_n]$ of ${\mathbb{R}}^n$, restrict to the degree $d$ homogeneous component, and take its factor space by the degree $d$ homogeneous multiples of $x_1^2+\cdots+x_n^2$. 

The restriction of an irreducible $ O_n$-module to $ SO_n$ either stays irreducible, or is the sum of two non-isomorphic irreducibles (having the same dimension). The details are as follows (they can be found for example on 
page 164 in \cite{weyl}): 
If $h_1(\lambda)<n/2$,
then the restriction $W_{\lambda}:={\mathrm{Res}}^{ O_n}_{ SO_n}V_{\lambda}$ remains irreducible over the special orthogonal group $ SO_n$. 
Moreover, denoting by $\lambda^\circ$ the partition with $h_1(\lambda^\circ)=n-h_1(\lambda)$ and  
$h_i(\lambda^\circ)=h_i(\lambda)$ for $i>1$, we have that $V_{\lambda^\circ}$ is isomorphic to 
the tensor product of $V_{\lambda}$ and the determinant representation of $ O_n$, hence the restriction to $ SO_n$ of $V_{\lambda^\circ}$ is also isomorphic to $W_{\lambda}$. 
When $n=2l+1$ is odd, then 
$\{W_{\lambda}\mid h_1(\lambda)\leq n/2\}$ is a complete list of isomorphism classes of irreducible $ SO_n$-modules. 
When $n=4m$ is divisible by four and 
$h_1(\lambda)=n/2$, then ${\mathrm{Res}}^{ O_n}_{ SO_n}V_{\lambda}$ decomposes as the direct sum 
$W_{\lambda}\oplus W_{(\lambda_1,\ldots, \lambda_{2m-1},-\lambda_{2m})}$ of two non-isomorphic irreducible 
$ SO_n$-modules having the same dimension, and 
$\{W_{(\lambda_1,\ldots,\lambda_{2m})}\mid \lambda_1\geq\cdots\geq\lambda_{2m-1}\geq |\lambda_{2m}|\}$ 
is a complete list of isomorphism classes of irreducible representations of $ SO_n$. 
When $n=4m+2$ and $h_1(\lambda)=n/2$, then ${\mathrm{Res}}^{ O_n}_{ SO_n}V_{\lambda}$ remains irreducible over $ SO_n$, and 
$\{W_{\lambda}\mid h_1(\lambda)<n/2\}\cup\{{\mathrm{Res}}^{ O_n}_{ SO_n}V_{\lambda}\mid h_1(\lambda)=n/2\}$  
is a complete list of isomorphism classes of irreducible $ SO_n$-modules.  

Although in our problem we are dealing with real representations of real Lie groups, in order to 
study concrete representations we shall apply the so-called {\it highest weight theory}, and therefore we shall need to change to representations of the complex orthogonal groups 
$O_n({\mathbb{C}}) :=\{A\in{\mathbb{C}}^{n\times n}\mid A^TA=I\}$ and $ SO_n({\mathbb{C}}):=\{A\in O_n({\mathbb{C}})  \mid \det(A)=1\}$. 
The passage is as folows: 
First recall that  a {\it complex representation of a real Lie group $G$} is a (real) Lie group homomorphism from $G$ into the group of invertible linear transformations of some finite dimensional complex vector space. For any representation of a real Lie group $G$ on some finite dimensional real vector space $V$ there is an associated complex representation of $G$ (called {\it its complexification}): namely, consider the induced ${\mathbb{C}}$-linear action of $G$ on 
${\mathbb{C}}\otimes_{{\mathbb{R}}}V$. 
The complexification of an irreducible $ O_n$-module or $ SO_n$-module 
stays irreducible, with the exception that when $n=4m+2$ and $h_1(\lambda)=n/2$, then the 
restriction to $ SO_n$ of the complexification of the irreducible 
$ O_n$-module $V_{\lambda}$ splits as the sum 
$W_{\lambda}+W_{(\lambda_1,\ldots,\lambda_{2m},-\lambda_{n/2})}$of two 
non-isomorphic equidimensional irreducible complex representions 
of $ SO_n$ (just like as it happens already over the reals when $n=4m$). 
Next recall that representations of $ O_n$ or $ SO_n$ are {\it polynomial}, that is, the matrix elements  
of a representation are polynomials in the matrix entries of the elements of our group.  
Therefore the complexification of a representation on $V$ extends to a polynomial representation of $O_n({\mathbb{C}}) $ or $ SO_n({\mathbb{C}})$ on the complexified vector space ${\mathbb{C}}\otimes_{{\mathbb{R}}}V$. This extension is unique (since the equations defining our groups 
inside the space of $n\times n$ matrices are the same in the complex and the real cases). 
Given an irreducible $ O_n$-module (or $ SO_n$-module) $V_{\lambda}$ or $W_{\lambda}$, we keep the same symbol to denote the corresponding irreducible polynomial representations of 
the corresponding complex group $O_n({\mathbb{C}}) $ or $ SO_n({\mathbb{C}})$. 
It is clear from the discussion above that given a representation of $ O_n$ or $ SO_n$ on $V$, the multiplicity of an irreducible representation $V_{\lambda}$ or $W_{\lambda}$ as a summand in $V$ is the same as the multiplicity of the corresponding irreducible representation of $O_n({\mathbb{C}}) $ or $ SO_n({\mathbb{C}})$ as a summand in ${\mathbb{C}}\otimes_{{\mathbb{R}}}V$. 

The so-called highest weight theory is a standard tool to decompose a given polynomial 
$ SO_n({\mathbb{C}})$-module as a sum of irreducibles. 
To apply highest weight theory it is convenient to perform a linear change of variables and work with the orthogonal group 
\begin{eqnarray*}O_n({\mathbb{C}},J)&:=&\{A\in{\mathbb{C}}^{n\times n}\mid A^TJA=J\}\\
 SO_n({\mathbb{C}},J)&:=&\{A\in O_n({\mathbb{C}},J)\mid \det(A)=1\}
 \end{eqnarray*} 
preserving the symmetric bilinear form on ${\mathbb{C}}^n$ with matrix $J$, where for $n=2l$ even we have 
$J=\left(\begin{array}{cc}0& I \\  I & 0\end{array}\right)$, a $2\times 2$  block matrix, with the $l\times l$ identity matrix  $I$ in the off-diagonal positions, and the zero matrix in the diagonal positions, and for $n=2l+1$ odd we have 
$J=\left(\begin{array}{ccc}0&I &0\\  I &0& 0\\ 0& 0& 1\end{array}\right)$. 
Denote by ${\mathbb{T}}$ the subgroup of $ SO_n({\mathbb{C}},J)$ consisting of the diagonal matrices 
\[\{{t=\mathrm{diag}}(t_1,\ldots,t_l,t_1^{-1},\ldots,t_l^{-1})\mid 
t_1,\ldots,t_l\in{\mathbb{C}}^\times\}\] 
when $n=2l$ and 
\[\{t={\mathrm{diag}}(t_1,\ldots,t_l,t_1^{-1},\ldots,t_l^{-1},1)\mid 
t_1,\ldots,t_l\in{\mathbb{C}}^\times\}\] 
when $n=2l+1$. Then  ${\mathbb{T}}$ is a maximal torus in $ SO_n({\mathbb{C}},J)$ (in the terminology of algebraic groups). 
Characters of ${\mathbb{T}}$ are identified with $l$-tuples of integers: given 
$\alpha=(\alpha_1,\ldots,\alpha_l)\in{\mathbb{Z}}^l$ 
and $t\in {\mathbb{T}}$ as above we write $\alpha(t):=\prod_{i=1}^lt_i^{\alpha_i}$. 
An element $v$ in an $ SO_n({\mathbb{C}},J)$-module $V$ is called a weight vector if for some character $\alpha$ of ${\mathbb{T}}$ we have $t\cdot v=\alpha(t)v$ ($t\in {\mathbb{T}}$); in this case we call $\alpha$ the {\it weight} of 
$v$. 
Denote by $u^+$ the unipotent radical given for example in section 10.4.1 in \cite{procesi}  of the positive Borel subalgebra  of the Lie algebra ${\mathrm{so}}_n({\mathbb{C}},J)$ of $ SO_n({\mathbb{C}},J)$. 
A non-zero element $w$ in an $ SO_n({\mathbb{C}},J)$-module $W$ is called a {\it highest weight vector of weight} 
$\lambda=(\lambda_1,\ldots,\lambda_l)$ if $w$ is annihilated by $u^+$ 
(the Lie algebra ${\mathrm{so}}_n({\mathbb{C}})$ acts on $V$ via the tangent representation of the given representation of $ SO_n({\mathbb{C}})$), and $t\cdot w=\lambda(t)w$ for all $t\in{\mathbb{T}}$. Such a vector generates an irreducible $ SO_n({\mathbb{C}},J)$-submodule in $W$ isomorphic to $W_{\lambda}$. We recall that there is a standard partial ordering of weights in representation theory: the weight $\alpha$ is greater than the weight $\beta$ if $\alpha-\beta$ is a sum of positive roots of the Lie algebra ${\mathrm{so}}_n({\mathbb{C}},J)$.  Now $\lambda$ is the unique maximal element (with respect to this partial ordering) among the weights of ${\mathbb{T}}$ that occur in $W_{\lambda}$.


\section{Spherical harmonics in the vanishing ideal of degenerate matrices} \label{sec:spherical} 

Denote by ${\mathcal{M}}_{{\mathbb{C}}}$ the space of $n\times n$ complex symmetric matrices.  As a module over 
$O_n({\mathbb{C}}) $ it decomposes as ${\mathcal{M}}_{{\mathbb{C}}}={\mathcal{N}}_{{\mathbb{C}}}\oplus{\mathbb{C}} I$, where ${\mathcal{N}}_{{\mathbb{C}}}$ is the subspace of trace zero $n\times n$ complex symmetric matrices. 
View ${\mathcal{N}}_{{\mathbb{C}}}$ as a complex affine algebraic variety with  coordinate ring 
${\mathbb{C}}[{\mathcal{N}}_{{\mathbb{C}}}]$. 
The same formulae as in (\ref{eq:H_i})  and 
(\ref{eq:T}) give an $O_n({\mathbb{C}}) $-equivariant polynomial map 
${\mathcal{T}}_{{\mathbb{C}}}:{\mathcal{N}}_{{\mathbb{C}}}\to\bigwedge^{n-1}{\mathcal{N}}_{{\mathbb{C}}}$, 
and we want to decompose the image of the dual of $\bigwedge^{n-1}{\mathcal{N}}_{{\mathbb{C}}}$ under the comorphism ${\mathcal{T}}_{{\mathbb{C}}}^\star:{\mathbb{C}}[\bigwedge^{n-1}{\mathcal{N}}_{{\mathbb{C}}}]\to {\mathbb{C}}[{\mathcal{N}}_{{\mathbb{C}}}]$ of ${\mathcal{T}}$. 
Clearly ${\mathcal{N}}_{{\mathbb{C}}}$ contains the subsets ${\mathcal{N}}\supset{\mathcal{F}}$. 
Denote by ${\mathcal{F}}_{{\mathbb{C}}}$ the closure of ${\mathcal{F}}$ in the Zariski topology of the complex affine space ${\mathcal{N}}_{{\mathbb{C}}}$, and denote by ${\mathcal{I}}({\mathcal{F}}_{{\mathbb{C}}})$ the vanishing ideal  
in ${\mathbb{C}}[{\mathcal{N}}_{{\mathbb{C}}}]$ of ${\mathcal{F}}_{{\mathbb{C}}}$.  Note that ${\mathcal{I}}({\mathcal{F}}_{{\mathbb{C}}})$ is spanned over ${\mathbb{C}}$ by its real subspace ${\mathcal{I}}({\mathcal{F}})$, and 
${\mathcal{I}}({\mathcal{F}}_{{\mathbb{C}}})\supset {\mathcal{T}}_{{\mathbb{C}}}^\star((\bigwedge^{n-1}{\mathcal{N}}_{{\mathbb{C}}})^\star)$ spanned over 
${\mathbb{C}}$ by ${\mathcal{T}}^\star((\bigwedge^{n-1}{\mathcal{N}})^\star)$. As explained in Section~\ref{sec:repsofonr}, 
the $ O_n$-module strucure of ${\mathcal{T}}^\star((\bigwedge^n{\mathcal{N}})^\star)$ and 
${\mathcal{I}}({\mathcal{F}})$ can be read off form the $O_n({\mathbb{C}}) $-module structure of 
${\mathcal{T}}_{{\mathbb{C}}}^\star((\bigwedge^{n-1}{\mathcal{N}}_{{\mathbb{C}}})^\star)$ and ${\mathcal{I}}({\mathcal{F}}_{{\mathbb{C}}})$, so from now on we shall focus on the complex objects. (Let us stress explicitly that 
${\mathcal{F}}_{{\mathbb{C}}}$ is properly contained in the set of all complex trace zero symmetric matrices with multiple eigenvalues; the latter is an irreducible complex hypersurface in  ${\mathcal{N}}_{{\mathbb{C}}}$, namely the set of all complex zeros of the discriminant, whereas ${\mathcal{F}}_{{\mathbb{C}}}$ is a codimension two complex algebraic subvariety of ${\mathcal{N}}_{{\mathbb{C}}}$.)  

As we indicated in Section~\ref{sec:repsofonr}, we change to the groups 
$O_n({\mathbb{C}},J)$ and $ SO_n({\mathbb{C}},J)$ preserving the symmetric bilinear form on ${\mathbb{C}}^n$ with matrix $J$.  
Accordingly, ${\mathcal{M}}_{{\mathbb{C}}}$ has to be replaced by the space 
\[{\mathcal{M}}_{{\mathbb{C}},J}:=\{A\in{\mathbb{C}}^{n\times n}\mid A^T=JAJ^{-1}\}\] 
of self-adjoint linear operators on $({\mathbb{C}}^n,J)$, on which 
$O_n({\mathbb{C}},J)$ acts by conjugation: for $g\in O_n({\mathbb{C}},J)$ and $A\in{\mathcal{M}}_{{\mathbb{C}},J}$ we have 
$g\cdot A=gAg^{-1}$ (matrix multiplication on the right hand side). 
The space 
${\mathcal{N}}_{{\mathbb{C}}}$ of trace zero symmetric matrices has to be replaced by 
\[{\mathcal{N}}_{{\mathbb{C}},J}:=\{A\in{\mathcal{M}}_{{\mathbb{C}},J}\mid {\mathrm{Tr}}(A)=0\}.\] 
Note that if $K$ is an $n\times n$ matrix with $J=K^TK$, then 
conjugation by $K^{-1}$ gives  isomorphisms 
$g\mapsto K^{-1}gK$, $O_n({\mathbb{C}}) \to O_n({\mathbb{C}},J)$ and $A\mapsto K^{-1}AK$, ${\mathcal{M}}_{{\mathbb{C}}}\to {\mathcal{M}}_{{\mathbb{C}},J}$ that intertwines the actions of the orthogonal groups. Taking this into account it is easy to see  that the $O_n({\mathbb{C}},J)$-equivariant polynomial map ${\mathcal{T}}_{{\mathbb{C}},J}:{\mathcal{N}}_{{\mathbb{C}},J}\to 
\bigwedge^{n-1}{\mathcal{N}}_{{\mathbb{C}},J}$  corresponding to ${\mathcal{T}}_{{\mathbb{C}}}$ is given by the same formulae  as in (\ref{eq:H_i}) and (\ref{eq:T}).  

\begin{proposition}\label{prop:highestweight} 
The $ SO_n({\mathbb{C}},J)$-module  ${\mathcal{T}}_{{\mathbb{C}},J}^\star((\bigwedge^{n-1}{\mathcal{N}}_{{\mathbb{C}},J})^\star)$ 
contains a summand isomorphic to $W_{(n)}$. 
\end{proposition} 

\begin{proof} Denote $x_{ij}$ the function on ${\mathcal{N}}_{{\mathbb{C}},J}$ mapping an $n\times n$ matrix in 
${\mathcal{N}}_{{\mathbb{C}},J}$ to its $(i,j)$-entry. For a diagonal matrix $t\in{\mathbb{T}}$ (cf. Section~\ref{sec:repsofonr}) we have that $t\cdot x_{ij}$ is the $(i,j)$-entry of $t^{-1}(x_{ij})_{i,j=1}^nt$ (matrix multiplication). 
In particular, all the $x_{ij}$ are weight vectors in ${\mathcal{N}}_{{\mathbb{C}},J}^\star$  (and $x_{l+1,1}$ is the unique highest weight vector in ${\mathcal{N}}_{{\mathbb{C}},J}^\star$, it has weight $(2)$).  

Consequently, 
\[x:=x_{2,1}\wedge x_{3,1}\wedge \cdots \wedge x_{n,1}\in \bigwedge^{n-1}{\mathcal{N}}_{{\mathbb{C}},J}^\star\] 
is a weight vector, and one computes easily that its weight is $(n)$. 
 Moreover, it is a highest weight vector, since its weight is maximal with respect to the lexicographic ordering among the weights that occur in $\bigwedge^{n-1}{\mathcal{N}}_{{\mathbb{C}},J}^\star$ 
 (one can easily see that for all other weights $\alpha=(\alpha_1,\ldots,\alpha_l)$ in $\bigwedge^{n-1}{\mathcal{N}}_{{\mathbb{C}},J}^\star$ we have $\alpha_1<n$), hence it is also maximal with respect to the standard partial ordering of weights mentioned in the end of Section~\ref{sec:repsofonr}. 

Use the usual identification $\bigwedge^{n-1}({\mathcal{N}}_{{\mathbb{C}},J}^\star)\cong 
(\bigwedge^{n-1}{\mathcal{N}}_{{\mathbb{C}},J})^\star$: for $x_1,\ldots,x_{n-1}$ in ${\mathcal{N}}^\star_{{\mathbb{C}},J}$ and 
$A_1,\ldots,A_{n-1}$ in ${\mathcal{N}}_{{\mathbb{C}},J}$, the value of  $x_1\wedge \cdots \wedge x_{n-1}$ (viewed as a linear form on $\bigwedge^{n-1}{\mathcal{N}}_{{\mathbb{C}},J}$) at  
$A_1\wedge \cdots \wedge A_{n-1}$ equals 
\[\sum_{\pi\in {\mathrm{Sym}}(n-1)}{\mathrm{sign}}(\pi)x_1(A_{\pi(1)})\cdots 
x_{n-1}(A_{\pi(n-1)})\] 
(where the summation ranges over the full symmetric group ${\mathrm{Sym}}(n-1)$ of degree $n-1$). 
In particular, this means that the value of ${\mathcal{T}}_{{\mathbb{C}},J}^\star(x)$ on $A\in{\mathcal{N}}_{{\mathbb{C}},J}$ equals 
the determinant of the $(n-1)\times (n-1)$ matrix, whose $i$th column is  the first column (with the $(1,1)$-entry removed) of $A^i$. 
Now take for $A$ the matrix of the linear transformation permuting the standard basis vectors  $e_1,\ldots,e_n\in{\mathbb{C}}^n$ cyclically as follows: 
\[e_1\mapsto e_{l+1}\mapsto e_{l+2}\mapsto\cdots\mapsto e_n\mapsto e_l\mapsto e_{l-1}\mapsto 
\ldots\mapsto e_2\mapsto e_1\] 
(for $n=4$ the matrix $A$ is displayed in the proof of Theorem~\ref{thm:n=4}). 
It is easy to see that $A$ belongs to ${\mathcal{N}}_{{\mathbb{C}},J}$. The first columns (with the first entry removed) of the first $n-1$ powers of 
$A$  exhaust the set of standard basis vectors in ${\mathbb{C}}^{n-1}$, showing that 
${\mathcal{T}}_{{\mathbb{C}},J}^\star(x)(A)\neq 0$. 
Consequently, ${\mathcal{T}}_{{\mathbb{C}},J}^\star(x)$ is non-zero, and so it is a highest weight vector of weight 
$(n)$, generating an $ SO_n({\mathbb{C}},J)$-submodule  isomorphic to $W_{(n)}$. 
\end{proof} 

\begin{theorem}\label{thm:spherical} 
The degree $n(n-1)/2$ homogeneous component of  the vanishing ideal 
${\mathcal{I}}({\mathcal{F}})$ of degenerate trace zero symmetric $n\times n$ real matrices contains an $ SO_n$-submodule isomorphic to ${\mathcal{H}}^n({\mathbb{R}}^n)$, the space of $n$-variable spherical harmonics 
of degree $n$. 
Consequently, the discriminant of $n\times n$ symmetric matrices can be written as the sum of 
$\binom{2n-1}{n-1}-\binom{2n-3}{n-1}$ squares. 
\end{theorem}

\begin{proof} As explained in Section~\ref{sec:repsofonr} and in the beginning of 
Section~\ref{sec:spherical}, Proposition~\ref{prop:highestweight} 
implies the first statement. 
The second statement follows by Lemma~\ref{lemma:irred}.  
\end{proof} 

\begin{remark} Based on Kummer's result $\mu(3)\leq 7=\dim({\mathcal{H}}^3({\mathbb{R}}^3))$, 
Peter Lax surmized the inequality 
$\mu(n)\leq  \binom{2n-1}{n-1}-\binom{2n-3}{n-1}$ in his letter \cite{lax:2009} to the author.   
This is a drastic improvement compared to the general upper bounds for 
$\mu(n)$ appearing in prior work known to us (cf. Remark~\ref{remark:tracezero}). 
On the other hand, 
this inequality is not always sharp (see Sections~\ref{sec:n=3} and  \ref{sec:n=4}).  
\end{remark}


\section{The case $n=3$}\label{sec:n=3}

\begin{proposition}\label{prop:idealcubic} 
For $n=3$, the degree three homogeneous component ${\mathcal{I}}({\mathcal{F}})_3$ of the vanishing ideal of degenerate trace zero symmetric matrices is isomorphic to the seven dimensional irreducible $SO_3$-module ${\mathcal{H}}^3({\mathbb{R}}^3)$, and coincides with the image under the map ${\mathcal{T}}^\star$ of $(\bigwedge^2{\mathcal{N}})^\star$.  
\end{proposition} 

\begin{proof} 
Set $K:=\left(\begin{array}{ccc}\frac{1}{\sqrt 2} & \frac{1}{\sqrt 2} & 0 \\\frac{{\mathrm{i}}}{\sqrt 2} & 
\frac{-{\mathrm{i}}}{\sqrt 2} & 0 \\0 & 0 & 1\end{array}\right)$. 
Then we have $K^TK=J$. 
As explained in Sections~\ref{sec:repsofonr} and \ref{sec:spherical}, it is sufficient to prove that the degree three homogeneous component of ${\mathcal{I}}({\mathcal{F}}_{{\mathbb{C}},J})$ is isomorphic as an $SO_3({\mathbb{C}},J)$-module  to $W_{(3)}$, and coincides with ${\mathcal{T}}_{{\mathbb{C}},J}^\star((\bigwedge^2{\mathcal{N}}_{{\mathbb{C}}})^\star)$, 
where 
\[{\mathcal{F}}_{{\mathbb{C}},J}:=\{K^{-1}AK\mid A\in{\mathcal{F}}_{{\mathbb{C}}}\}\] 
(so ${\mathcal{F}}_{{\mathbb{C}},J}$ is the Zariski closure in ${\mathcal{N}}_{{\mathbb{C}},J}$ of the subset 
$K^{-1}\cdot{\mathcal{F}}\cdot K$). 
The character of the $SO_3({\mathbb{C}},J)$-module ${\mathcal{N}}_{{\mathbb{C}},J}$ (i.e. the trace of the group element 
${\mathrm{diag}}(t,t^{-1},1)$ as a linear operator on ${\mathcal{N}}_{{\mathbb{C}},J}$) equals 
$t^2+t+1+t^{-1}+t^{-2}$, hence the character of $\bigwedge^2{\mathcal{N}}_{{\mathbb{C}},J}$ is 
\[t^3+t^2+2t+2+2t^{-1}+t^{-2}+t^{-3}.\] 
Since the character of $W_{(d)}$ is $\sum_{j=-d}^d t^j$, we conclude that 
\[(\bigwedge^2{\mathcal{N}}_{{\mathbb{C}},J})^\star\cong\bigwedge^2{\mathcal{N}}_{{\mathbb{C}},J}\cong W_{(1)} + W_{(3)}.\] 
The first summand is the defining representation of $SO_3({\mathbb{C}},J)$ on ${\mathbb{C}}^3$, 
and it is isomorphic to the adjoint representation on ${\mathrm{so}}_3({\mathbb{C}},J)$. It follows from the considerations in Section~\ref{sec:kercovstar} that the kernel of the map $\kappa$ 
(defined in Section~\ref{sec:kercovstar}) is isomorphic to the irreducible $SO_3({\mathbb{C}},J)$-module $W_{(3)}$, hence 
${\mathcal{T}}_{{\mathbb{C}},J}^\star((\bigwedge^2{\mathcal{N}}_{{\mathbb{C}},J})^\star)\cong W_{(3)}$ (say by 
the special case $n=3$ of Theorem~\ref{thm:spherical}). 

The degree three homogeneous component of ${\mathbb{C}}[{\mathcal{N}}_{{\mathbb{C}},J}]$ is isomorphic as an 
$SO_3({\mathbb{C}},J)$-module to the third symmetric tensor power ${\mathrm{S}}^3(W_{(2)})$.
Calculating its character  
we get 
\[{\mathbb{C}}[{\mathcal{N}}_{{\mathbb{C}},J}]_3\cong W_{(6)}+ W_{(4)}+W_{(3)}+ W_{(2)}+W_{(0)}.\] 
Next we determine the highest weight vectors of the irreducible summands in the above decomposition. The unipotent radical of the positive Borel subalgebra of the Lie algebra 
${\mathrm{so}}_3({\mathbb{C}},J)$ is one-dimensional spanned by $E:=E_{13}-E_{32}$, where $E_{ij}$ stands for  the $3\times 3$ matrix unit whose only non-zero entry is a $1$ in the $(i,j)$-position 
(see for example Section 10.4.1 in \cite{procesi}).  
The following table gives the effect of $E$ on the basis elements $x_{ij}\in{\mathcal{N}}_{{\mathbb{C}},J}^\star$ 
(the coordinate function 
mapping a matrix in ${\mathcal{N}}_{{\mathbb{C}},J}$ to its $(i,j)$-entry), as well as the weights of the $x_{ij}$. 
To compute it note that the action of the Lie algebra ${\mathrm{so}}_3({\mathbb{C}},J)$ is the following: 
$A\in{\mathrm{so}}_3({\mathbb{C}},J)$ maps $B\in{\mathcal{N}}_{{\mathbb{C}},J}$ to $[A,B]:=AB-BA$ (matrix multiplication on the right hand side). Consequently, $E$ sends $x_{ij}\in{\mathcal{N}}_{{\mathbb{C}},J}^\star$  to the $(i,j)$-entry of 
$[(x_{ij})_{i,j=1}^3,E]$. The matrix 
$H:=E_{11}-E_{22}$ spans the Cartan subalgebra of ${\mathrm{so}}_3({\mathbb{C}},J)$, and 
the weight of $x_{ij}$ is $k\in{\mathbb{Z}}$ if $H$ maps $x_{ij}$ to $kx_{ij}$. 
Note that we have the following linear relations in ${\mathcal{N}}_{{\mathbb{C}},J}^\star$: 
$x_{22}=x_{11}$, $x_{33}=-2x_{11}$, $x_{13}=x_{32}$, $x_{23}=x_{31}$. From this information one easily works out the following table: 

\begin{center}
\begin{tabular}{c||c|c|c|c|c}
$x\in {\mathcal{N}}_{{\mathbb{C}},J}^\star$ & $x_{21}$ & $x_{11}$ & $x_{31}$ & $x_{12}$ & $x_{32}$  \\
\hline 
weight of $x$ & $2$ & $0$ & $1$ & $-2$ & $-1$ 
\\ \hline 
$E(x)$ & $0$ & $-x_{31}$ & $x_{21}$ & $-2 x_{32}$ & $3 x_{11}$ 
\end{tabular}
\end{center}

The coefficient of $t^2$ in the character of ${\mathrm{S}}^3({\mathcal{N}}_{{\mathbb{C}},J}^\star)$ is $4$, hence 
the weight subspace of weight $2$ in ${\mathrm{S}}^3({\mathcal{N}}_{{\mathbb{C}},J}^\star)$ is $4$-dimensional. 
From the above table one easily sees that 
$x_{21}x_{11}^2, x_{21}x_{31}x_{32},x_{21}^2x_{12}, x_{11}x_{31}^2$ is a basis of this weight space, and computes the images under $E$ of these weight vectors. For example, 
$E(x_{21}x_{11}^2)=E(x_{21})x_{11}^2+x_{21}E(x_{11})x_{11}+x_{21}x_{11}E(x_{11})=
-2x_{21}x_{11}x_{31}$. 
Solving a system of linear equations  one gets explicitly the elements in this weight space annihilated by 
$E\in{\mathrm{so}}_3({\mathbb{C}},J)$. 
It turns out that (up to non-zero scalar multiples) there is one highest weight vector of weight $2$ in ${\mathrm{S}}^3({\mathcal{N}}_{J,{\mathbb{C}}}^\star)$. One finds similarly all the highest wight vectors in the degree three homogeneous component of the coordinate ring of ${\mathcal{N}}_{{\mathbb{C}},J}$, they are listed in the following  table (the highest weight vector of weight $0$ is a scalar multiple of the function 
$A\mapsto {\mathrm{Tr}}(A^3)$). 

\begin{center}
\begin{tabular}{c|c} 
weight & highest weight vector in ${\mathrm{S}}^3({\mathcal{N}}_{{\mathbb{C}},J}^\star)$ \\ \hline \hline 
$6$ & $x_{21}^3$ \\ \hline 
$4$ & $2x_{21}^2x_{11}+x_{21}x_{31}^2$ \\ \hline 
$3$ & $x_{31}^3+3x_{21}x_{31}x_{11}-x_{21}^2x_{32}$ \\ \hline 
$2$ & $3x_{21}x_{11}^2+2x_{21}x_{31}x_{32}+x_{21}^2x_{12}$ \\ \hline 
$0$ & $-2x_{11}^3+2x_{11}x_{12}x_{21}-2x_{11}x_{32}x_{31}+x_{12}x_{31}^2+x_{32}^2x_{21}$
\end{tabular}
\end{center}

The diagonal matrix ${\mathrm{diag}}(2,-4,2)$ belongs to ${\mathcal{F}}$, hence 
\[K^{-1}\cdot {\mathrm{diag}}(2,-4,2)\cdot K=\left(\begin{array}{ccc}-1 & 3 & 0 \\ 3 & -1 & 0 \\ 0 & 0 & 2 \end{array}\right)\] belongs to ${\mathcal{F}}_{{\mathbb{C}},J}$. 
Direct computation shows that  there is only one polynomial in the second table vanishing on this matrix, namely the highest weight vector with weight $3$. 
Consequently, ${\mathcal{I}}({\mathcal{F}}_{{\mathbb{C}},J})_3\cong W_{(3)}$ as $SO_3({\mathbb{C}},J)$-modules. 
\end{proof}

\begin{proposition}\label{prop:multiplication}
Let $L$ denote the linear map from the symmetric tensor square of ${\mathcal{I}}({\mathcal{F}})_3$ into 
${\mathbb{R}}[{\mathcal{N}}]_6$ induced by the multiplication map 
${\mathcal{I}}({\mathcal{F}})_3\times{\mathcal{I}}({\mathcal{F}})_3\to {\mathbb{R}}[{\mathcal{N}}]_6$, $(f_1,f_2)\mapsto f_1f_2$. 
Then the kernel of $L$ is $5$-dimensional, and is isomorphic to the $SO_3$-module 
$W_{(2)}\cong {\mathcal{H}}^2({\mathbb{R}}^3)\cong {\mathcal{N}}$. 
\end{proposition}

\begin{proof} We shall compute explicit highest weight vectors in the symmetric tensor square 
${\mathrm{S}}^2({\mathcal{I}}({\mathcal{F}}_{{\mathbb{C}},J})_3)$, and select those that are mapped to zero under 
$L_{{\mathbb{C}},J}:{\mathrm{S}}^2({\mathcal{I}}({\mathcal{F}}_{{\mathbb{C}},J})_3)\to {\mathbb{C}}[{\mathcal{N}}_{{\mathbb{C}},J}]_6$, 
the  complexified version of $L$. 

Fix a highest weight vector $y_3$ in the irreducible $SO_3({\mathbb{C}},J)$-module $W_{(3)}$. 
So $E(y_3)=0$, and $y_3$ is uniquely determined  up to a non-zero scalar multiple. 
Then there is a unique basis $\{y_k\mid k=3,2,1,0,-1,-2,-3\}$ in $W_{(3)}$ such that 
$E(y_k)=y_{k+1}$ for $k=-3,-2,\ldots,2$. 
Set $F:=E_{31}-E_{23}$ and $H:=E_{11}-E_{22}$. 
Then $F$ spans the unipotent radical of the negative Borel subalgebra of the Lie algebra 
${\mathrm{so}}_3({\mathbb{C}},J)$, and $H$ spans the Cartan subalgebra of ${\mathrm{so}}_3({\mathbb{C}},J)$. 
Moreover, $H(y_k)=ky_k$, i.e. $y_k$ is a weight vector with weight $k$ for $k=-3,\ldots,3$. 
The relation $H=[E,F]=EF-FE$ shows that 
$y_2=\frac 13 F(y_3)$, $y_1=\frac 15 F(y_2)$, $y_0=\frac 16 F(y_1)$, $y_{-1}=\frac 16 F(y_0)$, 
$y_{-2}=\frac 15 F(y_{-1})$, $y_{-3}=\frac 13 F(y_{-2})$. Furthermore, $F(y_{-3})=0$. 
An easy character calculation yields 
\[{\mathrm{S}}^2(W_{(3)})\cong W_{(6)}+W_{(4)}+W_{(2)}+W_{(0)}.\] 
The highest weight vectors of the first three summands are 
\[w_{(6)}:=y_3^2, \quad w_{(4)}:=2y_3y_1-y_2^2, \quad w_{(2)}:=2y_3y_{-1}-2y_2y_0+y_1^2.\] 
Denote by $\iota$ the unique $SO_3({\mathbb{C}},J)$-module isomorphism 
$W_{(3)}\to{\mathcal{I}}({\mathcal{F}}_{{\mathbb{C}},J})_3$ mapping $y_3$ to the highest weight vector 
$x_{31}^3+3x_{21}x_{31}x_{11}-x_{21}^2x_{32}$ of ${\mathcal{I}}({\mathcal{F}}_{{\mathbb{C}},J})_3$ 
computed in the proof of Proposition~\ref{prop:idealcubic}. 
Keep the notation $\iota$ also for the induced isomorphism  
${\mathrm{S}}^2(W_{(3)})\to {\mathrm{S}}^2({\mathcal{I}}({\mathcal{F}}_{{\mathbb{C}},J})_3)$. The effect of $F$ on the variables $x_{ij}$ can be computed similarly to the first table in 
the proof of Proposition~\ref{prop:idealcubic}: 

\begin{center} 
\begin{tabular}{c||c|c|c|c|c}
$x\in{\mathcal{N}}_{{\mathbb{C}},J}^\star$ & $x_{21}$ & $x_{31}$ & $x_{11}$ & $x_{32}$ & $x_{12}$ 
\\ \hline 
$F(x)$ & $2x_{31}$ & $-3x_{11}$ & $x_{32}$ & $-x_{12}$ & $0$
\end{tabular} 
\end{center}  

Consequently, 
\[\iota(y_2)=\frac 13F\iota(y_3)=
\frac 13(-3x_{31}^2x_{11}-9x_{21}x_{11}^2-x_{31}x_{21}x_{32}+x_{21}^2x_{12})\]  
\[\iota(y_1)=\frac 15F(\iota(y_2))=\frac 13(-x_{31}^2x_{32}-3x_{21}x_{11}x_{32}+
x_{31}x_{21}x_{12})\]
\[\iota(y_0)=\frac 16 F\iota(y_1)=
\frac 16(x_{31}^2x_{12}-x_{21}x_{32}^2)\]
\[\iota(y_{-1})=\frac 16 F\iota(y_0)=
\frac 1{18}(-3x_{31}x_{11}x_{12}-x_{31}x_{32}^2+x_{21}x_{32}x_{12})\]

Now one gets by direct computation that 
$L_{{\mathbb{C}},J}\circ\iota (w_{(2)})=2\iota(y_3)\iota(y_{-1})-2\iota(y_2)\iota(y_0)+\iota(y_1)^2=0$, whereas 
$w_{(6)}$, $w_{(4)}$ do not belong to the kernel of $L_{{\mathbb{C}},J}\circ\iota$. 
Obviously $L_{{\mathbb{C}},J}\circ\iota (w_{(0)})$ is a non-zero scalar multiple of the discriminant 
(by Lemma~\ref{lemma:irred}), hence is non-zero. 
\end{proof} 

The linear and constant coefficients of the characteristic polynomial of 
the trace zero symmetric $3\times 3$ matrix  
\[\left(\begin{array}{ccc}a & d & e \\d & b & f \\e & f & c\end{array}\right)\quad\mbox{ where }
\quad c=-a-b\]
are 
\begin{eqnarray*}p&=&ab+ac+bc-d^2-e^2-f^2\\
&=& -a^2 - ab - b^2 - d^2 - e^2 - f^2
\end{eqnarray*}
and 
\begin{eqnarray*}q&=&-abc + af^2 + be^2 + cd^2 - 2def \\
&=& a^2b + ab^2 - ad^2 + af^2 - bd^2 + be^2 - 2def. 
\end{eqnarray*}

\begin{theorem}\label{thm:fivesquares} 
The discriminant $\delta=-4p^3-27q^2$ (where $p,q$ are given above) of $3\times 3$ trace zero symmetric matrices can be written as 
\begin{eqnarray*} 
\delta =&{}& 
27( aef - bef -de^2+ df^2)^2 \\ 
&+& 
(2a^3 + 3a^2b - 3ab^2 - ad^2 + 2ae^2- af^2 - 2b^3 + bd^2  + be^2 - 2bf^2)^2 \\ 
&+& 
(4a^2d + 10abd + 3aef + 4b^2d + 3bef - 2d^3 + de^2 + df^2)^2\\ 
&+& 
4(a^2e + abe+ 3adf  - 2b^2e+ 3bdf  - 2d^2e + e^3 + ef^2)^2 \\ 
&+& 
4(2a^2f  - abf -3ade- b^2f - 3bde + 2d^2f - e^2f - f^3)^2. 
\end{eqnarray*}
Moreover, $\delta=-4p^3-27q^2$ can not be written as the sum of four (or less) squares in 
${\mathbb{R}}[{\mathcal{N}}]={\mathbb{R}}[a,b,d,e,f]$,  
so $\mu(3)=5$. 
\end{theorem} 

\begin{proof} The formula can be checked by direct computation. 
We provide a derivation of it guided by representation theory;  
this approach leads also to a proof of the inequality $\mu(3)\geq 5$. 
Keep the notation introduced in the proof of Proposition~\ref{prop:multiplication}. 
Applying $F$ successively to $w_{(2)}$ we obtain the following basis in the summand $W_{(2)}$ of 
${\mathrm{S}}^2(W_{(3)})$: 
\begin{eqnarray*} & k_2:=2y_3y_{-1}-2y_2y_0+y_1^2,\quad 
k_1:=5y_3y_{-2}-3y_2y_{-1}+y_1y_0, \quad 
\\ & k_0:=5y_3y_{-3}-3y_1y_{-1}+2y_0^2, \quad 
\\ & k_{-1}:=5y_2y_{-3}-5y_1y_{-2}+2y_0y_{-1} \quad 
k_{-2}:=5y_1y_{-3}-10y_0y_{-2}+6y_{-1}^2.\end{eqnarray*} 
It is easy to see that the highest weight vector in the trivial summand $W_{(0)}$ of 
${\mathrm{S}}^2(W_{(2)})$ is 
\[w_{(0)}=2y_3y_{-3}-2y_2y_{-2}+2y_1y_{-1}-y_0^2.\]
Identify ${\mathrm{S}}^2(W_{(3)})\cong {\mathrm{S}}^2({\mathcal{I}}({\mathcal{F}}_{{\mathbb{C}},J})_3)$ via the isomorphism  
$\iota$ (see the proof of Proposition~\ref{prop:multiplication}). 
Now $L_{{\mathbb{C}},J}\circ\iota(w_{(0)})$ is a  non-zero scalar multiple of the discriminant  
$\delta\in{\mathbb{C}}[{\mathcal{N}}_{{\mathbb{C}},J}]$, since $\delta$  is the only degree six $SO_3({\mathbb{C}},J)$-invariant in ${\mathcal{I}}({\mathcal{F}}_{{\mathbb{C}},J})_6$. 
Moreover, $L_{{\mathbb{C}},J}\circ\iota(k_i)=0$ for $i=-2,-1,0,1,2$ by the proof of Proposition~\ref{prop:multiplication}. Consequently, there is a non-zero $c\in{\mathbb{C}}$ with 
\begin{equation}\label{eq:cdelta}
c\delta = L_{{\mathbb{C}},J}\circ\iota(5w_{(0)}-2k_0)
=L_{{\mathbb{C}},J}\circ\iota(-10y_2y_{-2}+16y_1y_{-1}-9y_0^2). 
\end{equation}
Complete the computation of the effect of $\iota$ on the basis of weight vectors of $W_{(3)}$ 
started in the proof of Proposition~\ref{prop:multiplication}: 
\[\iota(y_{-2})=\frac 15 F(\iota(y_{-1}))=
 \frac 1{90}(9x_{11}^2x_{12}+x_{31}x_{32}x_{12}+3x_{11}x_{32}^2-x_{21}x_{12}^2)\]
 \[\iota(y_{-3})=\frac 13 F(\iota(y_{-2}))=
 \frac 1{90}(3x_{11}x_{12}x_{32}-x_{31}x_{12}^2+x_{32}^3)\] 
 Denote by $\sigma:{\mathbb{C}}[{\mathcal{N}}_{{\mathbb{C}},J}]\to {\mathbb{C}}[{\mathcal{N}}_{{\mathbb{C}}}]={\mathbb{C}}[a,b,d,e,f]$ the 
 ${\mathbb{C}}$-algebra isomorphism given by 
 \begin{equation}\label{eq:sigma}
 \left(\begin{array}{ccc}\sigma(x_{11}) & \sigma(x_{12}) & \sigma(x_{13}) \\\sigma(x_{21}) & \sigma(x_{22}) & \sigma(x_{23}) \\\sigma(x_{31}) & \sigma(x_{32}) & \sigma(x_{33})\end{array}\right)
 =K^{-1}\left(\begin{array}{ccc}a & d & e \\d & b & f \\ e & f & -a-b\end{array}\right)K
\end{equation}
 where $K$ is the base change matrix given at the beginning of the proof of Proposition~\ref{prop:idealcubic}, so 
\begin{eqnarray*} 
\sigma(x_{11})&=&\frac 12(a+b), \\ 
\sigma(x_{21})=\frac 12(a-b)+{\mathrm{i}} \cdot d, &\quad &
\sigma(x_{12})=\frac 12(a-b)-{\mathrm{i}} \cdot d,\\  
\sigma(x_{31})=\frac 1{\sqrt 2}(e+{\mathrm{i}} \cdot f), &\quad &  
\sigma(x_{32})=\frac 1{\sqrt 2}(e-{\mathrm{i}} \cdot f). 
\end{eqnarray*} 
Note that since the characteristic polynomial of a matrix is conjugation invariant, it follows from 
(\ref{eq:sigma}) that the coefficients of the characteristic polynomial of $(x_{ij})_{3\times3}$ are mapped by $\sigma$ to the corresponding coefficients of the characteristic polynomial of 
the general symmetric trace zero matrix displayed before the statement of 
Theorem~\ref{thm:fivesquares}. Consequently, $\sigma$ maps $\delta\in{\mathbb{C}}[{\mathcal{N}}_{{\mathbb{C}},J}]$ to 
$\delta\in{\mathbb{C}}[{\mathcal{N}}_{{\mathbb{C}}}]$. 

Introduce the following operation on ${\mathbb{C}}[{\mathcal{N}}_{{\mathbb{C}}}]={\mathbb{C}}[a,b,d,e,f]$: a polynomial $h\in{\mathbb{C}}[{\mathcal{N}}_{{\mathbb{C}}}]$ 
can be uniquely written as ${\mathrm{Re}}(h)+{\mathrm{i}} \cdot{\mathrm{Im}}(h)$, where ${\mathrm{Re}}(h)$, ${\mathrm{Im}}(h)$ belong to the 
${\mathbb{R}}$-subalgebra ${\mathbb{R}}[{\mathcal{N}}]={\mathbb{R}}[a,b,d,e,f]$ of 
${\mathbb{C}}[{\mathcal{N}}_{{\mathbb{C}}}]$. Then set 
$\overline{h}:={\mathrm{Re}}(h)-{\mathrm{i}} \cdot{\mathrm{Im}}(h)$. 
Note that $h\cdot \overline{h}={\mathrm{Re}}(h)^2+{\mathrm{Im}}(h)^2$, and 
$h-\overline{h}=2{\mathrm{i}} \cdot{\mathrm{Im}}(h)$. 
Now observe that 
\[\sigma(x_{12})=\overline{\sigma(x_{21})}, \quad 
\sigma(x_{32})=\overline{\sigma(x_{31})}, \quad \mbox{and} \quad 
\overline{\sigma(x_{11})}=\sigma(x_{11}).\] 
Setting $z_j:=\sigma(\iota(y_j))$ for $j=\pm 1,\pm 2$, 
and $ z_0:=\frac 16\sigma(x_{31}^2x_{12})$, 
we have 
\[z_{-2}=-\frac 1{30}\overline{z_{2}},\quad 
z_{-1}=\frac 1{6}\overline{z_{1}},\quad  
\sigma(\iota(y_{0}))=z_0-\overline{z_0}. \]
Therefore by (\ref{eq:cdelta}) we get the following equality in ${\mathbb{C}}[{\mathcal{N}}_{{\mathbb{C}}}]$: 
\[c\delta=\frac 13({\mathrm{Re}}(z_2)^2+{\mathrm{Im}}(z_2)^2)
+\frac 83({\mathrm{Re}}(z_1)^2+{\mathrm{Im}}(z_1)^2)+36\cdot {\mathrm{Im}}(z_0)^2.\]
The right hand side is the sum of squares of five real polynomials; the constant turns out to be 
$\frac 1{4\cdot 27}$. Multiplying by $4\cdot 27$ the above equality one gets the formula in the  theorem. 

Next we show that the discriminant can not be written as the sum of four (or less) squares. 
Suppose to the contrary that $\delta=f_1^2+f_2^2+f_3^2+f_4^2$ for some $f_i\in{\mathbb{R}}[{\mathcal{N}}]$. 
Then all the $f_i$ are homogeneous of degree $3$, and all vanish on ${\mathcal{F}}$, hence 
$f_i\in{\mathcal{I}}({\mathcal{F}})_3$. Denote by $h_i$ the elements in the $SO_n({\mathbb{C}},J)$-module $W_{(3)}$ with 
$\sigma\circ\iota(h_i)=f_i$. 
Since $\sigma$ maps $\delta\in{\mathbb{C}}[{\mathcal{N}}_{{\mathbb{C}},J}]$ to $\delta\in{\mathbb{C}}[{\mathcal{N}}_{{\mathbb{C}}}]$, 
we have the equality $\sum_{i=1}^4\iota(h_i)^2=\delta\in{\mathbb{C}}[{\mathcal{N}}_{{\mathbb{C}},J}]$, i.e. 
$L_{{\mathbb{C}},J}\circ \iota (\sum_{i=1}^4h_i^2)=\delta\in {\mathbb{C}}[{\mathcal{N}}_{{\mathbb{C}},J}]$. 
Consequently, $c\sum_{i=1}^4h_i^2-w_{(0)}$ belongs to the kernel of $L_{{\mathbb{C}},J}\circ\iota$ for some non-zero $c\in{\mathbb{C}}$. 
By Proposition~\ref{prop:multiplication} 
there exist scalars $a_2,a_1,a_0,a_{-1},a_{-2}\in{\mathbb{C}}$ 
such that we have the following equality in ${\mathrm{S}}^2(W_{(3)})$: 
\begin{equation}\label{eq:foursquares} 
c(h_1^2+h_2^2+h_3^2+h_4^2)=w_{(0)}+2\sum_{j=-2}^2a_jk_j.
\end{equation}  
The choice of the basis $y_3,y_2,y_1,y_0,y_{-1},y_{-2},y_{-3}$ induces an identification between 
${\mathrm{S}}^2(W_{(3)})$ and the space of $7\times 7$ symmetric complex matrices. 
The element on the right hand side of (\ref{eq:foursquares}) 
corresponds to 
\[\left(\begin{array}{ccccccc}0 & 0 & 0 & 0 & 2a_2 & 5a_1 & 1+5a_0 \\0 & 0 & 0 & -2a_2 & -3a_1 & -1 & 5a_{-1} \\0 & 0 & 2a_2 & a_1 & 1-3a_0 & -5a_{-1} & 5a_{-2} \\0 & -2a_2 & a_1 & -1+4a_0 & 2a_{-1} & -10a_{-2} & 0 \\2a_2 & -3a_1 & 1-3a_0 & 2a_{-1} & 12a_{-2} & 0 & 0 \\5a_1 & -1 & -5a_{-1} & -10a_{-2} & 0 & 0 & 0 \\1+5a_0 & 5a_{-1} & 5a_{-2} & 0 & 0 & 0 & 0\end{array}\right).\]
 The left hand side of (\ref{eq:foursquares}) implies that the rank of the above matrix is at most four. 
 Therefore the determinant of the left upper $5\times 5$ minor is zero, hence $a_2=0$. 
 Similarly, the vanishing of the determinants of appropriate $5\times 5$ minors shows succesively 
 that $0=a_1=a_{-2}=a_{-1}$. Then the rank of our matrix is five if $1+5a_0$ or 
 $1-3a_0$ equals zero, the rank is six  if  $-1+4a_0=0$, and the rank is seven otherwise. This is a contradiction, hence 
 $\delta$ can not be written as the sum of four (or less) squares. 
\end{proof} 

\begin{remark}\label{remark:kummer} {\rm 
For comparison we give the expression for $\delta=-4p^3-27q^2$ as a sum of seven squares that is obtained  from the formula of  Kummer \cite{kummer} after restriction to ${\mathcal{N}}$: 
\begin{eqnarray*}\delta &=& 
15(ade + 2bde - d^2f + e^2f)^2 \\ 
&+& 
15( - 2adf - bdf +d^2e- ef^2)^2\\ 
&+&
15( aef - bef-de^2 + df^2)^2\\ 
&+& 
(- 4a^2f + 2abf +3ade + 2b^2f - d^2f - e^2f + 2f^3)^2 \\ 
&+& 
(2a^2e + 2abe - 4b^2e + 3bdf - d^2e + 2e^3 - ef^2)^2 \\ 
&+& 
(-4a^2d - 10abd - 3aef - 4b^2d  - 3bef + 2d^3 - de^2- df^2)^2\\ 
&+& 
(-2a^3 - 3a^2b + 3ab^2 + ad^2 - 2ae^2  + af^2+ 2b^3 - bd^2 - be^2 + 2bf^2)^2.
\end{eqnarray*} 
We mention that an alternative way to arrive at Kummer's  formula was given by 
Jacobi \cite{jacobi}. Computational aspects of the problem of writing a form as a sum of squares are discussed by Parrilo in \cite{parrilo}; in particular, using a method based on semidefinite programming Parrilo finds the same presentation for the $3\times 3$ discriminant as Kummer! 
Specializing $a\mapsto 0$, $b\mapsto 0$ in the above equality one recovers the expression for the discriminant $4(d^2+e^2+f^2)^3-108d^2e^2f^2$ of 
a symmetric $3\times 3$ matrix with zero diagonal entries 
as the sum of six squares found in \cite{lax:1998}. 
The specialization $a\mapsto 0$, $b\mapsto 0$ of the formula in Theorem~\ref{thm:fivesquares} 
yields 
\begin{eqnarray*}
4(d^2+e^2+f^2)^3-108d^2e^2f^2&=&
27(-de^2 + df^2)^2
+
(-2d^3 + de^2 + df^2)^2 \\
&+&
4(-2d^2e + e^3 + ef^2)^2
+
4(2d^2f - e^2f - f^3)^2
\end{eqnarray*}
(a sum of four squares on the right hand side). 
}\end{remark} 

\begin{remark}\label{remark:canbesmaller} {\rm
Comparing Proposition~\ref{prop:idealcubic} and Theorem~\ref{thm:fivesquares} we see 
that $\mu(n)$ can be strictly smaller than the minimal dimension of an irreducible 
$SO_n$-submodule in the degree $n(n-1)/2$ homogeneous component of the vanishing ideal of ${\mathcal{F}}$.}
\end{remark} 


\section{The case $n=4$}\label{sec:n=4} 

\begin{theorem}\label{thm:n=4} 
When $n=4$, the image of $(\bigwedge^3{\mathcal{N}})^\star$ under ${\mathcal{T}}^\star$ 
 in the degree $6$ homogeneous component of ${\mathcal{I}}({\mathcal{F}})$  is isomorphic as an $SO_4$-module to 
\[ W_{(3,3)}+W_{(3,-3)}+ W_{(4)}\] 
(the dimensions of the summands are $7,7,25$). 
Consequently, the discriminant of $4\times 4$ symmetric matrices can be written as the sum of seven squares. 
\end{theorem} 

\begin{proof} The second statement follows from the first by Lemma~\ref{lemma:irred}. 

To prove the first statement we may turn to the analogous statement for $SO_4({\mathbb{C}},J)$ and 
${\mathcal{N}}_{{\mathbb{C}},J}$  
(see the explanation in Sections~\ref{sec:repsofonr} and \ref{sec:spherical}).  
By a standard character calculation (the character of $W_{\lambda}$ is given for example in Section 24.2 of \cite{fulton-harris}) one obtains 
\[\bigwedge^3{\mathcal{N}}_{{\mathbb{C}},J}^\star\cong  W_{(3,3)}+W_{(3,-3)}+ W_{(4)}+W_{(3,1)}+W_{(3,-1)}+W_{(2)}+ W_{(1,1)} + W_{(1,-1)}\] 
as $ SO_4({\mathbb{C}},J)$-modules. Since $\dim(W_{\lambda})=(\lambda_1+1)^2-\lambda_2^2$ 
(see for example page 410 in \cite{fulton-harris}), the dimensions of the summands above are 
$7$, $7$, $25$, $15$, $15$, $9$, $3$, $3$. 
Since $W_{(1,1)}+ W_{(1,-1)}$ is isomorphic to the adjoint representation of $ SO_4({\mathbb{C}},J)$ on its Lie algebra ${\mathrm{so}}_4({\mathbb{C}},J)$, these two summands are annihilated by 
${\mathcal{T}}_{{\mathbb{C}},J}^\star$ by the considerations about the map $\kappa$ in Section~\ref{sec:kercovstar}.  
The summand $W_{(2)}\cong{\mathcal{N}}_{{\mathbb{C}},J}$ is also annihilated by ${\mathcal{T}}_{{\mathbb{C}},J}^\star$, see 
 Section~\ref{sec:kercovstar}. The result of Section~\ref{sec:spherical} in the special case 
 $n=4$ says that  $W_{(4)}$ occurs as a summand in the image of ${\mathcal{T}}_{{\mathbb{C}},J}^\star$. 

Now we shall find the highest weight vectors of weight $(3,3)$ respectively $(3,1)$ in 
$\bigwedge^3{\mathcal{N}}_{{\mathbb{C}},J}^\star$. 
The unipotent radical of the positive Borel subalgebra of ${\mathrm{so}}_4({\mathbb{C}},J)$ is spanned by 
$E_1:=E_{12}-E_{43}$ and $E_2:=E_{14}-E_{23}$, 
and the Cartan subalgebra of ${\mathrm{so}}_4({\mathbb{C}},J)$ is spanned by 
$H_1:=E_{11}-E_{33}$, $H_2:=E_{22}-E_{44}$ 
(see for example Section 10.4.1 in \cite{procesi}). 
Denote by $x_{ij}$ ($1\leq i,j\leq 4$) the usual coordinate functions on ${\mathcal{N}}_{{\mathbb{C}},J}$. 
Then $x_{11}$, $x_{12}$, $x_{13}$, $x_{14}$, $x_{24}$, $x_{21}$, $x_{31}$, $x_{41}$, 
$x_{42}$ is a basis in ${\mathcal{N}}_{{\mathbb{C}},J}^\star$, and we have the relations 
$-x_{22}=x_{33}=-x_{44}=x_{11}$, $x_{23}=x_{14}$, $x_{32}=x_{41}$, $x_{34}=x_{21}$, 
$x_{43}=x_{12}$. 
Recall that $ SO_4({\mathbb{C}},J)$ acts on ${\mathcal{N}}_{{\mathbb{C}},J}$ by conjugation. 
The tangent representation of the dual representation of $ SO_4({\mathbb{C}},J)$ 
on ${\mathcal{N}}_{{\mathbb{C}},J}^\star$ is the following representation of the Lie algebra ${\mathrm{so}}_4({\mathbb{C}},J)$: 
for a Lie algebra element $A\in{\mathrm{so}}_4({\mathbb{C}},J)$ and $x_{ij}\in{\mathcal{N}}_{{\mathbb{C}},J}^\star$ we have that  
$A(x_{ij})$ is the $(i,j)$ entry of the matrix commutator $XA-AX$, where $X=(x_{ij})_{i,j=1}^n$. 
Recall that given a representation of ${\mathrm{so}}_4({\mathbb{C}},J)$ on some vector space $V$, we say that $v\in V$ is a weight vector of weight $(\alpha_1,\alpha_2)\in {\mathbb{Z}}^2$ if 
$H_i(v)=\alpha_iv$ holds for $i=1,2$. In our case the $x_{ij}$ are all weight vectors. 
One gets the following table: 

\begin{center}
\begin{tabular}{c|c|c|c|c|c|c|c|c|c}
$x$ & $x_{11}$ & $x_{21}$ & $x_{31}$ & $x_{41}$ & $x_{42}$ &
$x_{12}$ & $x_{13}$ & $x_{14}$ & $x_{24}$ \\
\hline
weight & $0,0$ & $1,-1$& $2,0$ & $1,1$ & $0,2$ & $-1,1$ & $-2,0$ & $-1,-1$ & 
$0,-2$ \\ \hline 
$E_1(x)$ & $-x_{21}$ & $0$ & $0$ & $x_{31}$ & $2x_{41}$ & $2x_{11}$ & $-2x_{14}$ & 
$-x_{24}$ & $0$ \\ 
\hline 
$E_2(x)$ & $-x_{41}$ & $x_{31}$ & $0$ & $0$ & $0$ & $-x_{42}$ & $-2x_{12}$ & 
$2x_{11}$ & $2x_{21}$ 
\end{tabular}
\end{center} 

The $(3,3)$ weight space in $\bigwedge^3{\mathcal{N}}_{{\mathbb{C}},J}^\star$ is spanned by 
$x_{31}\wedge x_{41}\wedge x_{42}$, and this element is annihilated both by 
$E_1$ and $E_2$, as one can easily check using the table above. 
So this is a highest weight vector of weight $(3,3)$. 
Set 
\[A:=\left(\begin{array}{cccc}0 & 1 & 0 & 0 \\0 & 0 & 0 & 1 \\1 & 0 & 0 & 0 \\0 & 0 & 1 & 0\end{array}\right)\in{\mathcal{N}}_{{\mathbb{C}},J}, 
A^2=\left(\begin{array}{cccc}0 & 0 & 0 & 1 \\0 & 0 & 1 & 0 \\0 & 1 & 0 & 0 \\1 & 0 & 0 & 0\end{array}\right), 
A^3=\left(\begin{array}{cccc}0 & 0 & 1 & 0 \\1 & 0 & 0 & 0 \\0 & 0 & 0 & 1 \\0 & 1 & 0 & 0\end{array}\right).\] 
We have 
\[{\mathcal{T}}_{{\mathbb{C}},J}^\star(x_{31}\wedge x_{41}\wedge x_{42})(A)
=\det\left(\begin{array}{ccc}x_{31}(A) & x_{31}(A^2) & x_{31}(A^3) \\x_{41}(A) & x_{41}(A^2) & x_{41}(A^3) \\x_{42}(A) & x_{42}(A^2) &  x_{42}(A^3)\end{array}\right)=1\] 
so $x_{31}\wedge x_{41}\wedge x_{42}$ is not annihilated by 
${\mathcal{T}}_{{\mathbb{C}},J}^\star$, hence $W_{(3,3)}$ occurs as a summand in the image of ${\mathcal{T}}^\star$. 
Since this image is an $ O_4({\mathbb{C}},J)$-module, $W_{(3,-3)}$ is  also a summand in the image. 

The $(3,1)$ weight space is spanned by 
$x_{21}\wedge x_{31}\wedge x_{42}$ and $x_{11}\wedge x_{31}\wedge x_{41}$. 
Both are annihilated by $E_2$, whereas 
\[E_1(x_{21}\wedge x_{31}\wedge x_{42})=2x_{21}\wedge x_{31}\wedge x_{41},\quad 
E_1(x_{11}\wedge x_{31}\wedge x_{41})=-x_{21}\wedge x_{31}\wedge x_{41}.\]
Therefore there is one (up to non-zero scalar multiples) highest weight vector of weight $(3,1)$, namely 
$x_{21}\wedge x_{31}\wedge x_{42}+2 x_{11}\wedge x_{31}\wedge x_{41}$.  
We checked using {\hbox{\rm C\kern-.13em o\kern-.07em C\kern-.13em o\kern-.15em A}} \cite{cocoa} that ${\mathcal{T}}_{{\mathbb{C}},J}^\star$ maps this highest weight vector to zero. Consequently, $W_{(3,1)}$ is not a summand in the image of ${\mathcal{T}}_{{\mathbb{C}},J}^\star$. 
Since the image is an $ O_4({\mathbb{C}},J)$-submodule, $W_{(3,-1)}$ is not a summand in the image either. 
\end{proof} 

\begin{remark}\label{remark:borchardt} {\rm 
Borchardt \cite{borchardt} notes at the end of his paper that his general process for writing the discriminant as a sum of squares yields in the case $n=4$ an expression with $135$ summands, and states that the number of summands can be taken down to $84$. Moreover, he adds that the number $84$ may be further decreased, but does not specify the numbers.  }
\end{remark}

\begin{remark}\label{remark:mu(4)} {\rm 
Combining Theorems~\ref{thm:fivesquares} and ~\ref{thm:n=4} one gets 
$5\leq \mu(4)\leq 7$, since by Theorem 3 in \cite{lax:1998}, $\mu(n)$ is an increasing function of $n$. }
\end{remark}


\section{About the kernel of ${\mathcal{T}}^\star$}\label{sec:kercovstar}

We do not have a general formula for the multiplicities of the irreducible $SO_n$-summands in 
$ (\bigwedge^{n-1}{\mathcal{N}})^*\cong\bigwedge^{n-1}{\mathcal{N}}$. In Section~\ref{sec:spherical} we found an irreducible $SO_n$-submodule in  $(\bigwedge^{n-1}{\mathcal{N}})^*$ that is not mapped to zero under ${\mathcal{T}}^\star$. Here we present two constructions of some irreducible $O_n$-submodules 
(resp. $SO_n$-submodules) in the kernel of ${\mathcal{T}}^\star$.  

Denote by ${\mathrm{so}_n}$ the Lie algebra of $ SO_n$; it can be identified with the space of $n\times n$ skew-symmetric matrices, 
and the adjoint representation of $O_n$ is identified with the conjugation action. 
Moreover, the conjugation representation of $ O_n$ on the space ${\mathbb{R}}^{n\times n}$ of $n\times n$ matrices decomposes as 
${\mathbb{R}}^{n\times n}={\mathbb{R}} I\oplus {\mathcal{N}}\oplus {\mathrm{so}_n}$. 
We have an alternating multilinear $ O_n$-equivariant map 
\[{\mathcal{N}}\oplus \cdots\oplus{\mathcal{N}}\to{\mathbb{R}}^{n\times n},\quad (A_1,\ldots,A_{n-1})\mapsto 
\sum_{\pi\in {\mathrm{Sym}}(n-1)}{\mathrm{sign}}(\pi)A_{\pi(1)}\cdots A_{\pi(n-1)}\] 
where the summation above is over the full symmetric group ${\mathrm{Sym}}(n-1)$ of degree $n-1$. 
It induces an $ O_n$-module morphism 
\[\kappa:\bigwedge^{n-1}{\mathcal{N}} \to {\mathbb{R}}^{n\times n}.\] 
Denote by $\rho\in {\mathrm{Sym}}(n-1)$ the permutation $\rho(i)=n-i$ ($i=1,\ldots,n-1$). 
Then 
\[{\mathrm{sign}}(\rho)=
\begin{cases}-1, \quad &\mbox{  if  }n \mbox{ or }n-3 \mbox{ is divisible by } 4\\ 
 1, &\mbox{ if }n-1 \mbox{ or }n-2 \mbox{ is divisible by }4\end{cases}\]
The transpose of $\kappa(A_1\wedge\cdots\wedge A_{n-1})$ (where the $A_i$ are symmetric) is 
\[\sum_{\pi\in {\mathrm{Sym}}(n-1)}{\mathrm{sign}}(\pi)A_{\pi(n-1)}\cdots A_{\pi(1)}
={\mathrm{sign}}(\rho)\kappa(A_1\wedge\cdots\wedge A_{n-1})
\]
so ${\mathrm{im}}(\kappa)\subseteq {\mathcal{M}}$ when $n$ is congruent to $1$ or $2$ modulo $4$, whereas ${\mathrm{im}}(\kappa)\subseteq {\mathrm{so}_n}$ otherwise. 
Denoting by $E_{ij}$ the $n\times n$ matrix unit with the entry $1$ in the $(i,j)$ position and zeros everywhere else, we have 
\[\kappa((E_{12}+E_{21})\wedge (E_{23}+E_{32})\wedge \ldots \wedge (E_{n-1,n}+E_{n,n-1}))=
E_{1n}+{\mathrm{sign}}(\rho)E_{n1}.\] 
It follows that 
${\mathrm{im}}(\kappa)={\mathrm{so}_n}$  if  $n$  or $n-3$  is divisible by $4$ 
(by irreducibility of ${\mathrm{so}_n}$), 
so ${\mathrm{so}}_n$ is an $O_n$-module direct summand in $\bigwedge^{n-1}{\mathcal{N}}$, 
and ${\mathrm{im}}(\kappa)\supseteq{\mathcal{N}}$ 
when $n-1$ or $n-2$ is divisible by $4$ (in fact 
${\mathrm{im}}(\kappa)={\mathcal{N}}$ in the first case and 
${\mathcal{M}}$ in the second), so ${\mathcal{N}}$ is an $O_n$-module direct summand in $\bigwedge^{n-1}{\mathcal{N}}$. 

Since the matrices $H_i(A)$ pairwise commute for all $A\in{\mathcal{N}}$ (see (\ref{eq:H_i})), it follows that $\kappa\circ {\mathcal{T}} =0$, implying that  the $O_n$-module map ${\mathcal{T}}^\star$ factors through the surjection $(\bigwedge^{n-1}{\mathcal{N}})^\star\to\ker(\kappa)^\star$ (induced by the inclusion 
of $\ker(\kappa)$ into $\bigwedge^{n-1}{\mathcal{N}}$).  
Consequently, the summand ${\mathrm{so}_n}$ or ${\mathcal{N}}$ located above in 
$\bigwedge^{n-1}{\mathcal{N}} \cong (\bigwedge^{n-1}{\mathcal{N}})^\star$ is annihilated by 
${\mathcal{T}}^\star$, and 
 ${\mathcal{T}}^\star((\bigwedge^{n-1}{\mathcal{N}})^\star)$ is a non-zero homomorphic image of the $O_n$-module $\ker(\kappa)^\star\cong\ker(\kappa)$.

Next for $n=2l\geq 4$ even we  construct an $SO_n$-module surjection 
$\gamma:\bigwedge^{n-1}{\mathcal{N}}\to{\mathcal{N}}$; note that $\gamma$ is not $O_n$-equivariant. 
For an $n\times n$ skew-symmetric matrix $C$ denote by ${\mathrm{Pf}}(C)$ the pfaffian (see for example 
section 5.3.6 in \cite{procesi}). Note that for $A,B\in{\mathcal{N}}$ their commmutator 
$[A,B]=AB-BA$ is skew symmetric. Define the functions $F_{\alpha}$ on 
${\mathcal{N}}^n={\mathcal{N}}\oplus\cdots\oplus {\mathcal{N}}$ by 
\[{\mathrm{Pf}}(t_1[A_1,A_2]+\cdots+t_l[A_{n-1},A_n])=
\sum_{\alpha_1+\cdots+\alpha_l=l}t_1^{\alpha_1}\cdots t_l^{\alpha_l}F_{\alpha}(A_1,\ldots,A_n)\] 
where $t_1,\ldots,t_l$ are commuting indeterminates. 
Set $F:=F_{(1,\ldots,1)}$, so $F$ is an $n$-variable multilinear $SO_n$-invariant function on 
${\mathcal{N}}$. Define  
\[G(A_1,\ldots,A_n):=\frac{1}{n(n-2)\cdots 2}\sum_{\pi\in {\mathrm{Sym}}(n)}{\mathrm{sign}}(\pi)
F(A_{\pi(1)},\ldots,A_{\pi(n)})\]
an alternating $n$-variable multilinear $SO_n$-invariant on ${\mathcal{N}}$. 
The function $G$ is non-zero, since 
$G(B_1,C_1,\ldots,B_l,C_l)=1$ 
for the substitution 
\begin{equation}\label{eq:substitution} B_i:=E_{2i-1,2i-1}-E_{2i,2i}, \quad  
C_i:= \frac 12(E_{2i-1,2i}+E_{2i,2i-1}), \quad i=1,\ldots,l.
\end{equation}
Since $G$ is multilinear, it is naturally identified with 
\[\tilde G\in ({\mathcal{N}}\otimes\cdots\otimes{\mathcal{N}})^\star\cong {\mathcal{N}}^\star\otimes\cdots\otimes{\mathcal{N}}^\star.\] 
Identify the last tensor factor ${\mathcal{N}}^\star$ on the right hand side with ${\mathcal{N}}$ 
(using the trace form on ${\mathcal{N}}$), so view $\tilde  G$ as an element of 
\[\tilde G\in (({\mathcal{N}}^\star\otimes\cdots\otimes{\mathcal{N}}^\star)\otimes {\mathcal{N}})^{SO_n} \cong \hom_{SO_n}({\mathcal{N}}\otimes\cdots
\otimes {\mathcal{N}}, {\mathcal{N}}).\] 
Moreover, since $G$ is alternating, $\tilde G$ factors through the natural surjection 
${\mathcal{N}}\otimes\cdots\otimes{\mathcal{N}}\to\bigwedge^{n-1}{\mathcal{N}}$ and yields the desired non-zero element 
$\gamma\in\hom_{SO_n}(\bigwedge^{n-1}{\mathcal{N}},{\mathcal{N}})$. 
 It is easy to see that $\gamma\circ{\mathcal{T}}=0$: indeed, the commutator of any two of 
 $H_i(A)$, $i=1,\ldots,n-1$ (see (\ref{eq:H_i})) is zero, hence $F_{(1,\ldots,1)}$ becomes zero under a substitution of the arguments in any order by $A,H_2(A),\ldots,H_{n-1}(A), B$ 
 (where $A,B\in{\mathcal{N}}$ are arbitrary).  So $\gamma^\star$ embeds ${\mathcal{N}}^\star\cong{\mathcal{N}}$ as an $SO_n$-module direct summand in the kernel of ${\mathcal{T}}^\star$. 
 
 When $n-2$ is divisible by $4$, denote by $\kappa_1$ the composition of $\kappa$ and the projection 
 ${\mathrm{im}}(\kappa)={\mathcal{M}}={\mathcal{N}}\oplus {\mathbb{R}} I\to {\mathcal{N}}$. Then $\gamma$ and $\kappa_1$ are both $SO_n$-module surjections from $\bigwedge^{n-1}{\mathcal{N}}$ to ${\mathcal{N}}$. 
 However, they are not scalar multiples of each other, since $\kappa_1$ is $O_n$-equivariant, whereas $\gamma$ is not. (Alternatively,  
 $\kappa(B_1\wedge C_1\wedge\cdots \wedge B_l)=0$, where $B_i,C_j$ were defined in 
 (\ref{eq:substitution}), whereas 
 $\gamma(B_1\wedge C_1\wedge\cdots \wedge B_l)\neq 0$ as we pointed out above.) 
 Consequently,  the irreducible $SO_n$-module ${\mathcal{N}}$ appears with multiplicity $\geq 2$ as a summand in $\bigwedge^{n-1}{\mathcal{N}}$ when $n-2$ is divisible by $4$ (and $n>2$).

\begin{center} {\bf Acknowledgements}\end{center}  

Our interest in this problem was awakened by  a talk given by P\'eter Lax at the R\'enyi Institute of Mathematics in January of 2009. We are indebted to P\'eter Lax for encouragement, inspiring discussions, some references and corrections.



\frenchspacing
\bibliographystyle{plain}

\end{document}